\documentclass[12pt,a4paper]{amsart}
\usepackage{amsfonts,color}
\usepackage{amsthm}
\usepackage{amsmath}
\usepackage{amscd}
\usepackage{amssymb}
\usepackage[latin2]{inputenc}
\usepackage{t1enc}
\usepackage[mathscr]{eucal}
\usepackage{indentfirst}
\usepackage{graphicx}
\usepackage{graphics}
\usepackage{pict2e}
\usepackage{epic}
\usepackage{url}
\usepackage{epstopdf}
\usepackage{comment}
\usepackage[backend=biber,style=numeric,sorting=anyt,maxbibnames=10]{biblatex}
\usepackage{tikz}
\usepackage{enumitem}

\numberwithin{equation}{section}
\usepackage[margin=2.6cm]{geometry}

\allowdisplaybreaks

\theoremstyle{plain}
\newtheorem{Th}{Theorem}[section]
\newtheorem{Lemma}[Th]{Lemma}
\newtheorem{Cor}[Th]{Corollary}
\newtheorem{Prop}[Th]{Proposition}

 \theoremstyle{definition}
\newtheorem{Def}[Th]{Definition}

\newtheorem{Rem}[Th]{Remark}
\newtheorem{?}[Th]{Problem}
\newtheorem{Ex}[Th]{Example}

\newcommand{\E}{\mathbb{E}}

\newcommand{\ia}{\mathrm{ia}}
\newcommand{\ea}{\mathrm{ea}}
\newcommand{\alt}{\mathrm{alt}}

\begin{document}

\title{Permutation Tutte polynomial}

\author{Csongor Beke}

\address{Trinity College, University of Cambridge, CB2 1TQ, United Kingdom \\
}

\email{bekecsongor@gmail.com}

\author{Gergely K\'al Cs\'aji}

\address{Institute of Economics, Centre for Economic and Regional Studies, Hungary, H-1097 Budapest, T\'{o}th K\'{a}lm\'{a}n u. 4 \\
}

\email{csaji.gergely@krtk.hun-ren.hu}

\author[P. Csikv\'ari]{P\'{e}ter Csikv\'{a}ri}

\address{HUN-REN Alfr\'ed R\'enyi Institute of Mathematics, H-1053 Budapest Re\'altanoda utca 13-15 \and ELTE: E\"{o}tv\"{o}s Lor\'{a}nd University \\ Mathematics Institute, Department of Computer
Science \\ H-1117 Budapest
\\ P\'{a}zm\'{a}ny P\'{e}ter s\'{e}t\'{a}ny 1/C}

\email{peter.csikvari@gmail.com}

\author{S\'ara Pituk}

\address{ELTE: E\"{o}tv\"{o}s Lor\'{a}nd University \\ H-1117 Budapest
\\ P\'{a}zm\'{a}ny P\'{e}ter s\'{e}t\'{a}ny 1/C}

\email{pituksari@gmail.com}

\thanks{P\'eter Csikv\'ari is supported by the MTA-R\'enyi Counting in  Sparse Graphs ``Momentum'' Research Group and by the Dynasnet ERC Synergy
project (ERC-2018-SYG 810115). Gergely K\'al Cs\'aji is supported by the Hungarian Scientific Research Fund, OTKA, Grant No. K143858 and by the Momentum Grant of the Hungarian Academy of Sciences, grant number 2021-1/2021. S\'ara Pituk is supported by the Ministry of Culture and Innovation and the National Research, Development and Innovation Office within the Quantum Information National Laboratory of Hungary (Grant No. 2022-2.1.1-NL-2022-00004)}



\begin{abstract} The classical Tutte polynomial is a two-variate polynomial $T_G(x,y)$ associated to graphs or more generally, matroids. In this paper, we introduce a polynomial $\widetilde{T}_H(x,y)$ associated to a bipartite graph $H$ that we call the permutation Tutte polynomial of the graph $H$. It turns out that $T_G(x,y)$ and  $\widetilde{T}_H(x,y)$ share many properties, and the permutation Tutte polynomial serves as a tool to study the classical Tutte polynomial. We discuss the analogues of Brylawsi's identities and Conde--Merino--Welsh type inequalities. In particular, we will show that if $H$ does not contain isolated vertices, then 
$$\widetilde{T}_H(3,0)\widetilde{T}_H(0,3)\geq \widetilde{T}_H(1,1)^2,$$
which gives a short proof of the analogous result of Jackson:
$$T_G(3,0)T_G(0,3)\geq T_G(1,1)^2$$
for graphs without loops and bridges. We also give improvement on the constant $3$ in this statement by showing that one can replace it with $2.9243$.
\end{abstract}

\maketitle

\section{Introduction}
For a graph $G=(V,E)$ with $v(G)$ vertices and $e(G)$ edges, the Tutte polynomial $T_G(x,y)$ is defined as
$$T_G(x,y)=\sum_{A\subseteq E}(x-1)^{k(A)-k(E)}(y-1)^{k(A)+|A|-v(G)},$$
where $k(A)$ denotes the number of connected components of the graph $(V,A)$, see \cite{tutte1954contribution}. There is a vast literature on the properties of the Tutte polynomial and its applications, for instance, \cite{brylawski1992tutte,crapo1969tutte,ellis2011graph,welsh1999tutte} or the book \cite{ellis2022handbook}.

This paper aims to introduce an auxiliary polynomial that helps study the Tutte polynomial and has properties that make it interesting even on its own. We call this new polynomial the \emph{permutation Tutte polynomial}. It is defined for every bipartite graph.

\begin{Def} \label{main-def}
Let $H=(A,B,E)$ be a bipartite graph. Suppose that $V(H)=[m]$. For a permutation $\pi:[m]\to [m]$, we say that a vertex $i\in A$ is internally active if
$$\pi(i)>\max_{j\in N_H(i)}\pi(j),$$
where the maximum over an empty set is set to be $-\infty$.
Similarly, we say that vertex $j\in B$ is externally active if
$$\pi(j)>\max_{i\in N_H(j)}\pi(i).$$
Let $\ia(\pi)$ and $\ea(\pi)$ be the number of internally and externally active vertices in $A$ and $B$, respectively.
Let
$$\widetilde{T}_H(x,y)=\frac{1}{m!}\sum_{\pi \in S_m}x^{\ia(\pi)}y^{\ea(\pi)}.$$
We will call $\widetilde{T}_H(x,y)$ the permutation Tutte polynomial of $H$.
\end{Def}

\begin{Def}
The coefficients of $\widetilde{T}_H(x,y)$ will be denoted by $t_{i,j}(H)$, that is,
$$\widetilde{T}_H(x,y)=\sum_{i,j} t_{i,j}(H)x^iy^j.$$
\end{Def}

\begin{Ex}\label{example}
Let $H=P_5$ be the path on $5$ vertices where $A$ consists of $3$, $B$ consists of $2$ vertices, respectively. Then
$$\widetilde{T}_{P_5}(x,y)=\frac{2}{15}x^3 + \frac{4}{15}x^2 + \frac{1}{3}xy + \frac{2}{15}y^2 + \frac{1}{15}x + \frac{1}{15}y.$$
\end{Ex}

The motivation behind Definition~\ref{main-def} is the following characterisation of the Tutte polynomial.

\begin{Th}[Tutte \cite{tutte1954contribution}] \label{ia-ea-characterization}
Let $G$ be a connected graph with $m$ edges. Label the edges with $1,2,\dots,m$ arbitrarily. In the case of a spanning tree $T$ of $G$, let us call an edge $e\in E(T)$ internally active if $e$ has the largest label among the edges in the cut determined by $T$ and $e$ by removing $e$ from $T$. Let us call an edge $e\notin E(T)$ externally active if $e$ has the largest label among the edges in the cycle determined by $T$ and $e$ by adding $e$ to $T$. Let $\mathrm{ia}(T)$ and $\mathrm{ea}(T)$ be the number of internally and externally active edges, respectively. Then
$$T_G(x,y)=\sum_{T\in \mathcal{T}(G)}x^{\mathrm{ia}(T)}y^{\mathrm{ea}(T)},$$
where the summation goes for all spanning trees of $G$.
\end{Th}

Theorem~\ref{ia-ea-characterization} was originally a definition for the Tutte polynomial \cite{tutte1954contribution}. This characterization of the Tutte polynomial immediately shows that the coefficients of the Tutte polynomial are non-negative. In this theorem, we are restricted to the same labelling of the edges for all spanning trees. For those who have never seen this definition before, it might be very surprising that the Tutte polynomial is independent of the actual choice of the labelling.

To explain the connection between $T_G(x,y)$ and $\widetilde{T}_H(x,y)$, we need the concept of the local basis exchange graph.

\begin{Def} The local basis exchange graph $H[T]$ of a graph $G=(V,E)$ with respect to a spanning tree  $T$ is defined as follows.
 The graph $H[T]$ is a bipartite graph whose vertices are the edges of $G$. One bipartite class consists of the edges of $T$, the other consists of the edges of $E\setminus T$, and we connect a spanning tree edge $e$ with a non-edge $f$ if $f$ is in the cut determined by $e$ and $T$, equivalently, $e$ is in the cycle determined by $f$ and $T$. (Clearly, this definition works for general matroids and their basis.)
\end{Def}

Figure 1 depicts a graph $G$ with a spanning tree $T$ and the bipartite graph $H[T]$ obtained from $T$. 
\bigskip

For a fixed labelling of the edges of $G$, we get a labelling of the vertices of $H[T]$, and the internally (externally) active edges of $G$ correspond to internally (externally) active vertices of $H[T]$, so the two definitions of internal and external activity are compatible. Taking all permutations of the edge labels and averaging out will correspond to averaging out the constant $T_G(x,y)$ on the level of $G$, and will lead to the definition of $\widetilde{T}_{H[T]}(x,y)$.
This gives the identity
$$T_G(x,y)=\sum_{T\in \mathcal{T}(G)}\widetilde{T}_{H[T]}(x,y),$$
where the summation is over all spanning trees of $G$ (see Lemma~\ref{conn} for further details). This identity is the starting point of several proofs of our theorems concerning the Tutte polynomial.

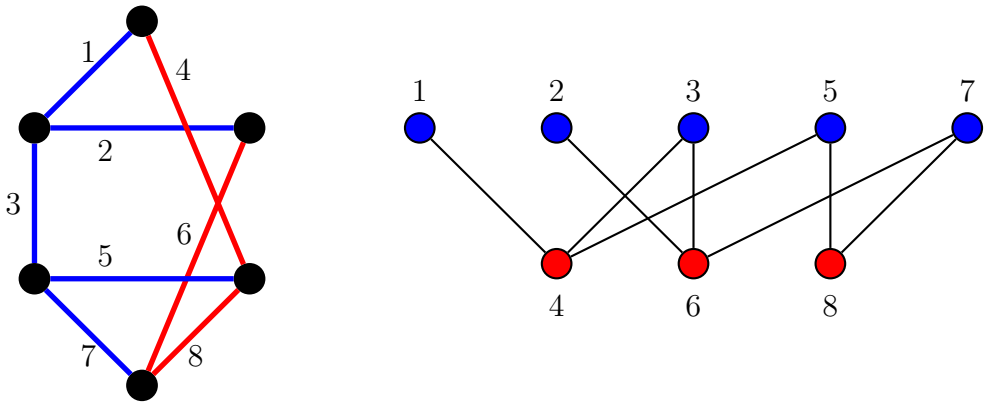
\begin{figure}[htp] 
\begin{tikzpicture}[, scale=0.33, baseline=0pt, node distance={20mm}, thick, main/.style = {draw, circle, fill=black}] 
\node[main] (1) {}; 
\node[main] (2) [above right of=1] {}; 
\node[main] (3) [below right of=2]{}; 
\node[main] (4) [below of=1]{}; 
\node[main] (5) [below of=3]{}; 
\node[main] (6) [below right of=4]{}; 
\draw [color=blue,line width=2pt](1) edge node[
above,black]{$1$} (2) ; 
\draw [color=blue, line width=2pt](1) edge node[pos=0.3, below, black]{$2$} (3) ; 
\draw [color=blue, line width=2pt](1) edge node[left, black]{$3$} (4) ; 
\draw [color=red, line width=2pt](2) edge  node[pos=0.15, right, black]{$4$} (5) ; 
\draw [color=red, line width=2pt](3) edge  node[pos=0.4, left, black]{$6$} (6) ; 
\draw [color=blue, line width=2pt](4) edge node[pos=0.3, above, black]{$5$} (5) ; 
\draw [color=blue, line width=2pt](4) edge  node[below, black]{$7$} (6) ; 
\draw [color=red, line width=2pt](5) edge  node[below, black]{$8$} (6) ; 
\end{tikzpicture} 
\qquad \qquad
\begin{tikzpicture}[, scale=0.33, baseline=0pt, node distance={18mm}, thick, main/.style = {draw, circle}]
\node[main, fill=blue, label=$1$] (1) {}; 
\node[main, fill=blue, label=$2$] (2) [right of=1]{};
\node[main,fill=blue, label=$3$] (3) [right of=2]{}; 
\node[main,fill=blue, label=$5$] (4) [right of=3]{}; 
\node[main,fill=blue, label=$7$] (5) [right of=4]{}; 
\node[main,fill=red, label={[yshift=-30pt]$4$}] (6) [below of=2]{}; 
\node[main,fill=red, label={[yshift=-30pt]$6$}] (7) [below of=3]{};
\node[main,fill=red, label={[yshift=-30pt]$8$}] (8) [below of=4]{};
\draw (1) -- (6) ; 
\draw (2) -- (7) ; 
\draw (3) -- (6) ; 
\draw (3) -- (7) ; 
\draw (4) -- (6) ; 
\draw (4) -- (8) ; 
\draw (5) -- (7) ; 
\draw (5) -- (8) ; 
\end{tikzpicture} 
\caption{Example for a graph $G$ and the local basis exchange graph $H[T]$ obtained from a spanning tree $T$.}
\end{figure}

In this paper, we use this machinery to study linear identities and inequalities.
In particular, we study the analogues of Brylawski's identities \cite{brylawski1992tutte}.
We also study a classical conjecture of Merino and Welsh \cite{merino1999forests} asserting that if $G$ is a graph without loops and bridges, then
$$\max(T_G(2,0),T_G(0,2))\geq T_G(1,1).$$
Conde and Merino \cite{conde2009comparing} gave the following strengthened version of this conjecture:
$$T_G(2,0)+T_G(0,2)\geq 2T_G(1,1)$$
and
$$T_G(2,0)T_G(0,2)\geq T_G(1,1)^2.$$
We will refer to these inequalities as the additive and multiplicative versions of the Conde--Merino--Welsh conjecture later.

Jackson \cite{jackson2010inequality} proved the following inequality:
$$T_G(3,0)T_G(0,3)\geq T_G(1,1)^2$$
for every graph $G$ without loops and bridges.
In this paper, we show that 
$$\widetilde{T}_H(3,0)\widetilde{T}_H(0,3)\geq \widetilde{T}_H(1,1)^2$$
holds for every bipartite graph $H$ without isolated vertices, 
 and this inequality implies Jackson's inequality (see the transfer-lemma, Lemma~\ref{quadratic-connection}). Our proof is completely different from the original proof of Jackson's inequality.
Our proof uses the Harris' inequality from probability theory and relies on the fact that permutations on $m$ elements can be generated by simply ordering $m$ random numbers chosen uniformly from $[0,1]$. This idea is the heart of several inequalities for $\widetilde{T}_H(x,y)$, and this is the key advantage of $\widetilde{T}_H(x,y)$ over $T_G(x,y)$.

As it turned out, the Merino--Welsh conjecture is not true for all matroids \cite{beke2024merino}, implying by the transfer-lemma that $\widetilde{T}_H(2,0)\widetilde{T}_H(0,2)\geq \widetilde{T}_H(1,1)^2$ is not true for all bipartite graphs without isolated vertices. Nevertheless, we also give several graph classes for which $\widetilde{T}_H(2,0)\widetilde{T}_H(0,2)\geq \widetilde{T}_H(1,1)^2$ holds, including complete bipartite graphs, regular bipartite graphs and trees. We also improve on Jackson's inequality by showing that
$$\widetilde{T}_H(x,0)\widetilde{T}_H(0,x)\geq \widetilde{T}_H(1,1)^2$$
for every bipartite graph $H$ without isolated vertices if $x\geq 2.9243$. By the transfer lemma, this implies that 
$$T_G(x,0)T_G(0,x)\geq T_G(1,1)^2$$
for every graph $G$ without loops and bridges (and matroids without loops and coloops) for the same values of $x$.
\bigskip

\noindent \textbf{Graphs and matroids.} 
 This is a remark about the Tutte polynomials of matroids.
In this paper, we mainly consider $T_G(x,y)$ for a graph $G$, but all our proofs work verbatim for the Tutte polynomial $T_M(x,y)$ of a matroid $M$. Furthermore, there are two results where it is more convenient to give immediately the statements for matroids instead of graphs, these are Corollary~\ref{simple-co-simple matroids} and Theorem~\ref{matr}. For this reason, we give a very brief account into the theory of matroids. 

A matroid $M$ is a pair $(E,\mathcal{I})$ such that $\mathcal{I}\subseteq 2^{E}$,  satisfying the axioms 
\begin{enumerate}[label=(\roman*)]
\item $\emptyset \in \mathcal{I}$, \item if $A'\subseteq A\in \mathcal{I}$, then $A'\in \mathcal{I}$, and \item if $A,B\in \mathcal{I}$ such that $|B|<|A|$, then there exists an $x\in A\setminus B$ such that $B\cup \{x\} \in \mathcal{I}$.
\end{enumerate}
The elements of $\mathcal{I}$ are called independent sets. Given a set $S\subseteq E$, the maximal independent subsets of $S$ all have the same cardinality, and this cardinality is called the rank of the set $S$, denoted by $r(S)$. The maximum size independent sets of $M$ are called bases, and their set is denoted by $\mathcal{B}(M)$. The size every basis of $M$ is the same, and this value is called the rank of $M$. The dual of a matroid $M$ is the matroid  $M^*$ whose bases are $\{E\setminus B\ |\ B\in \mathcal{B}(M) \}$. For further details on matroids, see for instance \cite{oxley1992matroid}

Given a graph $G=(V,E)$, the edge sets of the spanning forests of $G$ form the independent sets of a matroid $M_G$, called the cycle matroid of $G$. If $G$ is connected, then the bases of $M_G$ are the spanning trees of $G$. One can define the Tutte polynomial of a matroid 
 as $$T_M(x,y)=\sum_{S\subseteq E}(x-1)^{r(E)-r(S)}(y-1)^{|S|-r(S)},$$
where $r(S)$ is the rank of a set $S\subseteq E$. When $M=M_G$, then $T_{M_G}(x,y)=T_G(x,y)$. A loop in a matroid $M$ is an element $x\in E$ such that $r(\{x\})=0$, that is, $\{x\}\notin \mathcal{I}$. A coloop is an element that is a loop in the dual $M^*$ of the matroid $M$. Equivalently, a coloop is an element that is contained in every basis of $M$. For a cycle matroid $M_G$, loops correspond to loop edges and coloops correspond to bridges in the graph $G$. We call a matroid simple if $r(S)=|S|$ for $|S|\leq 2$, so there are no loops and there are no pairs $x,x'$ such that $\{x,x'\}\notin \mathcal{I}$, i.e., there are no parallel elements. We call a matroid cosimple if $M^*$ is simple.

Since we never use that our matroid comes from a graph, all our results hold even for general matroids.
\bigskip

\noindent \textbf{Notations.} 
The notation $[m]$ stands for the set $\{1,2,\dots ,m\}$.

Throughout this paper, $G$ will denote an arbitrary graph, and $H$ will denote a bipartite graph. For a vertex $v$, the graph $H-v$ is the graph obtained from $H$ by deleting $v$.
The graph $P_n$ denotes the path on $n$ vertices, and $C_n$ is the cycle on $n$ vertices. $K_n$ denotes the complete graph on $n$ vertices, and $K_{a,b}$ denotes the complete bipartite graph with classes of size $a$ and $b$.

For a bipartite graph $H=(A,B,E)$, we will also use the notations $A(H),B(H)$ for the bipartite classes if $H$ is not clear from the context. For any graph $G$, $v(G)$ denotes the number of vertices of $G$. When $G$ is an arbitrary graph, we generally use the notation $v(G)=n$, but when $H$ is a bipartite graph, then we use the notation $v(H)=m$, since $H$ often comes from a spanning tree $T$ of $G$, and in this case, the vertices of $H$ correspond to the edges of $G$.
\bigskip

\noindent \textbf{Organization of this paper.} In Section~\ref{sect: basic-recursions}, we establish some basic recursive formulas for $\widetilde{T}_H(x,y)$. Then in Section~\ref{sect: connection} we build the connection between $T_G(x,y)$ and $\widetilde{T}_H(x,y)$. In Section~\ref{sect: CMW inequalities}, we give some examples for Conde--Merino--Welsh-type inequalities for $T_G(x,y)$ and $\widetilde{T}_H(x,y)$. In Section~\ref{sect: FKG-inequalities},
we show how to apply Harris' inequality
to obtain several inequalities for $\widetilde{T}_H(x,y)$. In Section~\ref{sect: improvement 3}, we build on the previous section to improve on the constant in Jackson's inequality. In Section~\ref{sect: alternating function}, we study a special coefficient of the polynomial $\widetilde{T}_H(x,y)$. 
In Section~\ref{sect: Brylawski} we discuss Brylawski's identities for $\widetilde{T}_H(x,y)$ and $T_G(x,y)$. Finally, in Section~\ref{sect: concluding remarks} we include some remarks.

\section{Basic recursions} \label{sect: basic-recursions}

In this section, we establish several basic recursive identities for the permutation Tutte polynomial that we will use subsequently.

\begin{Lemma}\label{comp}
If $H$ is the disjoint union of $H_1$ and $H_2$, then
$$\widetilde{T}_H(x,y)=\widetilde{T}_{H_1}(x,y)\widetilde{T}_{H_2}(x,y).$$
In particular, if $v\in A$ is an isolated vertex, then
$$\widetilde{T}_H(x,y)=x\widetilde{T}_{H-v}(x,y).$$
Similarly, if $v\in B$ is an isolated vertex, then
$$\widetilde{T}_H(x,y)=y\widetilde{T}_{H-v}(x,y).$$
\end{Lemma}

\begin{proof} Assume that $H$ has $m$ vertices and let $\widetilde{T}_H(x,y)=\sum_{i,j}t_{i,j}(H)x^iy^j$.
Let us consider the number of permutations of $V(H)$ such that there are $i$ internally active vertices in $A$ and $j$ externally active vertices in $B$. By the above notation the number of these permutations is $m!t_{i,j}(H)$. We sort these permutations by looking at the number of internally and externally active vertices that are in $H_1$, and the number of internally and externally active vertices that are in $H_2.$ The number of permutations $\pi \in S_m$ such that $\ia_{H_1}(\pi) = i_1, \ia_{H_2}(\pi )= i_2, \ea_{H_1}(\pi )=j_1$ and $\ea_{H_2}(\pi )=j_2$ is
$$
m_1! t_{i_1,j_1} (H_1) m_2! t_{i_2,j_2}(H_2) \frac{m!}{m_1! m_2!},
$$ where $m_1 = v(H_1)$ and $m_2 = v(H_2)$. This holds, as to get such a permutation, we first have to decide which values will be assigned to the vertices of $H_1$ and which to $H_2$, and then assign these two sets of values to the two vertex sets independently, in a way such that in $H_k$ we have $i_k$ internally and $j_k$ externally active vertices $(k=1,2)$. Note that a set of $m_k$ distinct numbers can be identified with the set $[m_k]$, so this assignment within each subgraph can be viewed as a permutation of its vertex set. Thus we get
$$
m! t_{i,j}(H) = m! \sum_{\substack{i_1, i_2, j_1, j_2 \\ i_1+i_2=i \\ j_1+j_2=j}} t_{i_1,j_1}(H_1) t_{i_2,j_2}(H_2).
$$
Lemma \ref{comp} now follows from the definition of the product of two polynomials.
\end{proof}

The next lemma is the analogue of the property $T_M(x,y)=T_{M^*}(y,x)$ of the Tutte polynomial, where $M^*$ is the dual matroid of $M$.

\begin{Lemma} \label{csere}
For a bipartite graph $H=(A,B,E)$, let $H'=(B,A,E)$ be the graph obtained by switching the two sides of $H$. Then
$$
\widetilde{T}_H(x,y) = \widetilde{T}_{H'}(y,x).
$$
\end{Lemma}

\begin{proof}

For any $\pi\in S_m$, we have $\ia_H(\pi)=\ea_{H'}(\pi)$ and $\ea_H(\pi)=\ia_{H'}(\pi)$, so

$$\widetilde{T}_H(x,y)=\frac{1}{m!}\sum_{\pi \in S_m}x^{\ia_H(\pi)}y^{\ea_H(\pi)}=\frac{1}{m!}\sum_{\pi \in S_m}x^{\ea_{H'}(\pi)}y^{\ia_{H'}(\pi)}=\widetilde{T}_{H'}(y,x).$$

\end{proof}

\begin{Lemma} \label{rek}
If $H$ is a bipartite graph on $m$ vertices that does not contain isolated vertices, then
$$\widetilde{T}_H(x,y)=\frac{1}{m}\sum_{v\in V(H)}\widetilde{T}_{H-v}(x,y).$$
\end{Lemma}

\begin{proof}
For $\pi\in S_m$, let $v(\pi)$ be the vertex of $H$ such that $\pi(v(\pi))=1$. Let $\alpha(\pi)$ be the permutation of $V(H-v(\pi))$ where $\alpha(x)<\alpha(y)$ iff $\pi(x)<\pi(y)$. Then a vertex is internally (externally) active in $\alpha$ if and only if it is internally (externally) active in $\pi$, since $v(\pi)$ cannot be active as $v(\pi)$ is not isolated. Therefore $\ia_H(\pi)=\ia_{H-v(\pi)}(\alpha(\pi))$ and $\ea_H(\pi)=\ea_{H-v(\pi)}(\alpha(\pi))$. As $\pi$ runs through $S_m$, we remove each vertex $v\in V(H)$ exactly $(m-1)!$ times and get each permutation $\alpha$ of $\mathrm{Sym}\left([m]\backslash\{v\}\right)$ exactly once, so

\begin{align*}
    \widetilde{T}_H(x,y)&=\frac{1}{m!}\sum_{\pi \in S_m}x^{\ia_H(\pi)}y^{\ea_H(\pi)}\\
    &=\frac{1}{m!}\sum_{\pi \in S_m}x^{\ia_{H-v(\pi)}(\alpha(\pi))}y^{\ea_{H-v(\pi)}(\alpha(\pi))}\\
    &=\frac{1}{m!}\sum_{v\in V(H)}\sum_{\alpha \in Sym\left([m]\backslash\{v\}\right)}x^{\ia_{H-v}(\alpha)}y^{\ea_{H-v}(\alpha)}\\
    &=\frac{1}{m}\sum_{v\in V(H)}\widetilde{T}_{H-v}(x,y).
\end{align*}

\end{proof}

\section{Connection with the Tutte polynomial} \label{sect: connection}

In this section, we establish the main connection between the Tutte polynomial and the permutation Tutte polynomial. This connection will enable us to transfer linear identities and inequalities from the permutation Tutte polynomial to the Tutte polynomial.

\begin{Lemma}\label{conn}
Let $G$ be a graph. For each spanning tree $T$ of $G$, let $H[T]$ be the local basis exchange graph with respect to $T$. Then
$$T_G(x,y)=\sum_{T\in \mathcal{T}(G)}\widetilde{T}_{H[T]}(x,y),$$
where the sum is over the set of spanning trees $\mathcal{T}(G)$ of $G$.
\end{Lemma}

\begin{proof}
For a fixed spanning tree $T$ and a permutation $\pi$ of the edges, the internally and externally active edges correspond to the internally and externally active vertices of $H[T]$. Hence
$$T_G(x,y)=\sum_{T\in \mathcal{T}(G)}x^{\ia_{H[T]}(\pi)}y^{\ea_{H[T]}(\pi)}.$$
Now averaging it for all permutations $\pi \in S_m$ we get that
\begin{align*}
T_G(x,y)&=\frac{1}{m!}\sum_{\pi \in S_m}T_G(x,y)\\
&=\frac{1}{m!}\sum_{\pi \in S_m}\sum_{T\in \mathcal{T}(G)}x^{\ia_{H[T]}(\pi)}y^{\ea_{H[T]}(\pi)}\\
&=\sum_{T\in \mathcal{T}(G)}\frac{1}{m!}\sum_{\pi \in S_m}x^{\ia_{H[T]}(\pi)}y^{\ea_{H[T]}(\pi)}\\
&=\sum_{T\in \mathcal{T}(G)}\widetilde{T}_{H[T]}(x,y).
\end{align*}

\end{proof}

\begin{Rem}
The local basis exchange graph $H[T]$ has an isolated vertex if and only if $G$ contains a bridge or a loop. Furthermore, $H[T]$ is connected if and only $G$ is $2$-connected.
\end{Rem}

The following lemma enables us to study Conde-Merino-Welsh type inequalities.

\begin{Lemma}[Transfer lemma] \label{quadratic-connection}
Let $x_0,x_1,x_2,y_0,y_1,y_2\geq 0$. Suppose that for any bipartite graph $H$, we have
$$\widetilde{T}_{H}(x_1,y_1)\widetilde{T}_{H}(x_2,y_2)\geq \widetilde{T}_{H}(x_0,y_0)^2.$$
Then for any graph $G$, we have
$$T_{G}(x_1,y_1)T_{G}(x_2,y_2)\geq T_G(x_0,y_0)^2.$$
More generally, if for $x_0,x_1,\dots ,x_n,y_0,y_1,\dots ,y_n\geq 0$ and $\alpha_1,\dots ,\alpha_n\geq 0$ satisfying $\sum_{k=1}^n\alpha_k=1$, the inequality
$$\prod_{k=1}^n\widetilde{T}_{H}(x_k,y_k)^{\alpha_k}\geq \widetilde{T}_{H}(x_0,y_0)$$
holds true for every bipartite graph $H$, then for every graph $G$, we have
$$\prod_{k=1}^nT_{G}(x_k,y_k)^{\alpha_k}\geq T_{G}(x_0,y_0).$$
\end{Lemma}

\begin{proof}
We have
\begin{align*}
T_G(x_1,y_1)T_G(x_2,y_2)&=\left( \sum_{T\in \mathcal{T}(G)}\widetilde{T}_{H[T]}(x_1,y_1)\right)\left( \sum_{T\in \mathcal{T}(G)}\widetilde{T}_{H[T]}(x_2,y_2) \right)\\
&\geq\left(\sum_{T\in \mathcal{T}(G)}\left(\widetilde{T}_{H[T]}(x_1,y_1)\widetilde{T}_{H[T]}(x_2,y_2)\right)^{1/2}\right)^2\\
&\geq \left(\sum_{T\in \mathcal{T}(G)}\widetilde{T}_{H[T]}(x_0,y_0)\right)^2\\
&=T_G(x_0,y_0)^2.
\end{align*}
The first and last equality are the applications Lemma~\ref{conn}. The first inequality is a Cauchy--Schwarz inequality applied to the numbers $\widetilde{T}_{H[T]}(x_1,y_1)^{1/2},\widetilde{T}_{H[T]}(x_2,y_2)^{1/2}$ for $T\in \mathcal{T}(G)$. This is where we use that $x_1,x_2,x_3,y_1,y_2,y_3\geq 0$ to ensure that we can consider the square roots. The second inequality is simply the condition of the lemma.

The proof of the more general statement follows the same way, the only difference is that Cauchy--Schwarz inequality we have to use the following version of H\"older's inequality:
$$
\prod_{k=1}^n\left( \sum_{j=1}^Ma_{kj}\right)^{\alpha_k}\geq \sum_{j=1}^M\prod_{k=1}^na_{kj}^{\alpha_k}.$$
\end{proof}

\section{Conde-Merino-Welsh type inequalities}
\label{sect: CMW inequalities}

In this section, we study inequalities of type
$$\widetilde{T}_H(x_1,y_1)\widetilde{T}_H(x_2,y_2)\geq \widetilde{T}_H(x_0,y_0)^2$$
and
$$T_G(x_1,y_1)T_G(x_2,y_2)\geq T_G(x_0,y_0)^2.$$
As Lemma~\ref{quadratic-connection} shows, the former inequality implies the latter one. Hereafter we refer to these types of inequalities as Conde-Merino-Welsh type inequalities.

In this section, we collect two simple results. The first one is motivated by a result of Merino, Iba{\~n}ez and Rodr{\'\i}guez \cite{merino2009note} and implies their result by the transfer lemma (Lemma~\ref{quadratic-connection}). The proof is almost the same as their proof.

\begin{Lemma}\label{lem442}
If $H$ does not contain isolated vertices, then
$$\widetilde{T}_H(4,0)\widetilde{T}_H(0,4)\geq \widetilde{T}_H(2,2)^2.$$
\end{Lemma}

\begin{proof}
By definition, $\widetilde{T}_H(x,y) = \alt(H) x^a + f(x,y) = \alt(H) y^b + g(x,y)$, where the coefficients of $f$ and $g$ are non-negative. This implies that
$$\widetilde{T}_H(4,0) \geq \alt(H) 4^a\ \ \ \mbox{and}\ \ \ \widetilde{T}_H(0,4) \geq \alt(H) 4^b.$$ By Lemma \ref{22}, we have $\widetilde{T}_H(2,2) =\alt(H) 2^{a+b}$, so
$$\widetilde{T}_H(4,0) \widetilde{T}_H(0,4) \geq \alt(H)^2 4^{a+b} = \widetilde{T}_H(2,2)^2.$$
\end{proof}

Our next goal is to show that a complete bipartite graph $H$ satisfies\\ $\widetilde{T}_H(2,0)\widetilde{T}_H(0,2)\geq \widetilde{T}_H(1,1)^2$. Note that $\widetilde{T}_H(1,1)=1$ for every bipartite graph $H$.

\begin{Lemma}\label{lem:complete}
For the complete bipartite graph $K_{a,b}$ with $m=a+b$, we have
$$\widetilde{T}_{K_{a,b}}(x,y)=\sum_{i=1}^a\frac{a(a-1)...(a-i+1)b}{m(m-1)...(m-i)}x^i+\sum_{j=1}^b\frac{b(b-1)...(b-j+1)a}{m(m-1) \dots (m-j)}y^j.$$
\end{Lemma}

\begin{proof}
If $i>0$, which means that there is an internally active vertex $v$ in $A$, then $\pi(v)$ is greater than $\pi(w)$ for every $w \in B$. This means that there cannot be any externally active vertex in $B$, so $j$ must be $0$. Similarly, if $j>0$, then $i=0$. If we want to count the number of permutations $\pi$ such that $\ia(\pi)=0$ and $\ea(\pi)=j$, we have to consider all the permutations such that $\pi^{-1}(1), \pi^{-1}(2), \dots, \pi^{-1}(j)$ are in $B$, but $\pi^{-1}(j+1)$ is in $A$. The case $j=0$ is similar.
\end{proof}

\begin{Prop}
Let $a,b\geq 1$, then
$$\widetilde{T}_{K_{a,b}}(2,0)\widetilde{T}_{K_{a,b}}(0,2)\geq \widetilde{T}_{K_{a,b}}(1,1)^2.$$
\end{Prop}

\begin{proof}
We first check the statement if $\min(a,b)=1$. We can assume that $a=1,b=m-1$. Then
$$\widetilde{T}_{K_{1,m-1}}(x,y)=\frac{1}{m}\left(y^{m-1}+y^{m-2}+\dots +y+x\right),$$
and so
$$\widetilde{T}_{K_{1,m-1}}(2,0)\widetilde{T}_{K_{1,m-1}}(0,2)=\frac{2}{m}\frac{2^{m-1}+\dots +2}{m}\geq 1=\widetilde{T}_{K_{a,b}}(1,1)^2$$
if $m\geq 2$. Now we prove the statement by induction on $m$. The case $m=2$ is trivial. Suppose that we already know the statement holds till $m-1$. We can assume that $\min(a,b)\geq 2$. Then for $H=K_{a,b}$ we have
\begin{align*}
\widetilde{T}_H(2,0)\widetilde{T}_H(0,2)&=\left(\frac{1}{m}\sum_{v\in V}\widetilde{T}_{H-v}(2,0)\right)\left(\frac{1}{m}\sum_{v\in V}\widetilde{T}_{H-v}(0,2)\right)\\
&\geq \frac{1}{m^2}\left(\sum_{v\in V}\left(\widetilde{T}_{H-v}(2,0)\widetilde{T}_{H-v}(0,2)\right)^{1/2}\right)^2\\
&\geq \frac{1}{m^2}\left(\sum_{v\in V}1\right)^2\\
&=1\\
&=\widetilde{T}_{K_{a,b}}(1,1)^2.\\
\end{align*}
In the first step, we used the recursion formula for $\widetilde{T}_H(x,y)$. In the second step, we used a Cauchy--Schwarz inequality. In the third step, we used that $H-v$ is also a complete bipartite graph without isolated vertices. This completes the induction step and the proof. 
\end{proof}

\begin{Rem} \label{complete-bipartite-linear}
We remark that
$$\widetilde{T}_{K_{r,r}}(2,0)+\widetilde{T}_{K_{r,r}}(0,2)=\widetilde{T}_{K_{r,r}}(2,2)=\alt(K_{r,r})2^{2r}=\frac{2^{2r}}{\binom{2r}{r}}\approx \sqrt{r\pi}.$$
The first equality follows from the fact that $K_{r,r}$ cannot contain active vertices on both sides. The second equality is Lemma~\ref{22}. Using that $\widetilde{T}_{K_{r,r}}(2,0)=\widetilde{T}_{K_{r,r}}(0,2)$ this shows that  
$$\widetilde{T}_{K_{r,r}}(2,0)\widetilde{T}_{K_{r,r}}(0,2)\approx \frac{r\pi}{4}.$$
\end{Rem}

\subsection{Counter-examples.}

The paper \cite{beke2024merino} shows that there are matroids without loops and coloops for which $$T_M(2,0)T_M(0,2)< T_M(1,1)^2.$$
This immediately implies that $\widetilde{T}_H(2,0)\widetilde{T}_H(0,2)
\geq \widetilde{T}_H(1,1)^2$ cannot be true in general. In this section, we construct such bipartite graphs. A historical comment: the counter-example for the Merino--Welsh conjecture grew out from the counter-examples treated in this section.

\begin{Def}
For positive integers $a,b,c$ with $c\leq b$, let $H_{a,b,c}$ be the graph that we obtained from $K_{a,b}=(A,B,E)$ by attaching $c$ pendant vertices to $c$ distinct elements of $B$. So the resulting graph has $a+b+c$ vertices with $a+c$ and $b$ on the different sides.
\end{Def}

\begin{Rem}
In the paper \cite{beke2024merino}, the authors consider the matroid $U^{(2)}_{3k,2k}$, the $2$-thickening of the uniform matroid $U_{3k,2k}$ on $3k$ elements with rank $2k$. For this matroid, the local basis exchange graph is isomorphic to $H_{2k,2k,2k}$ for every basis. 
\end{Rem}

Let 
$$S(a,b,c)=\widetilde{T}_{H_{a,b,c}}(x,y).$$
Then using Lemma~\ref{rek}, we get that
$$S(a,b,c)=\frac{1}{a+b+c}\left(aS(a-1,b,c)+cxS(a,b-1,c-1)+(b-c)S(a,b-1,c)+cS(a,b,c-1)\right).$$
Together with the boundary conditions 
$$S(0,b,c)=\left(\frac{x+y}{2}\right)^{c}y^{b-c},$$
$$S(a,0,c)=x^{a+c},$$
$$S(a,b,0)=\widetilde{T}_{K_{a,b}}(x,y),$$
this function is completely determined and can be computed in a fast way.

For various choices of $(a,b,c)$, we got that 
$$\widetilde{T}_{H_{a,b,c}}(2,0)\widetilde{T}_{H_{a,b,c}}(0,2)<1.$$
For example, if $a=19,b=c=21$, then
$$\widetilde{T}_{H_{a,b,c}}(2,0)=\frac{17823568079808010514820609}{519645565199326904320}\approx 34299.4711654...$$
and 
$$\widetilde{T}_{H_{a,b,c}}(0,2)=\frac{205317845112145723813}{7322325659223715408773120}\approx 0.000028039977278...,$$
and their product is approximately
$0.961756392151...$

Notably, $(a,b,c)=(22,22,22)$ also provides a counter-example to the matroidal version of the Merino--Welsh conjecture.

\begin{Rem}
Let us call a matroid basis-equivalent if all the local basis exchange graphs are isomorphic. Clearly, basis-transitive matroids are such matroids. An interesting question is to determine which bipartite graphs $H$ can be the local basis exchange graph of a basis-equivalent matroid. Since $T_M(x,y)=c\widetilde{T}_H(x,y)$ in this case, then by comparing the coefficients of $x^a$, we get  a very strong necessary condition: $c=T_M(1,1)=\frac{1}{\alt(H)}\in \mathbb{Z}$ and $\frac{1}{\alt(H)}\widetilde{T}_H(x,y)$ has only integer coefficients. 
\end{Rem}

\section{Applications of correlation inequalities}
\label{sect: FKG-inequalities}

In this section, we show the advantage of $\widetilde{T}_H(x,y)$ over $T_G(x,y)$ in proving Conde-Merino-Welsh-type inequalities.

Let us immediately give two inequalities as motivations.

\begin{Lemma} \label{special rectangle}
Let $H$ be an arbitrary bipartite graph. Suppose that $0\leq x\leq 1$ and $y\geq 1$ or $0\leq y\leq 1$ and $x\geq 1$. Then
$$\widetilde{T}_H(x,y)\widetilde{T}_H(1,1)\geq \widetilde{T}_H(x,1)\widetilde{T}_H(1,y).
$$
If both $x,y\geq 1$ or both $0\leq x,y\leq 1$, then
$$\widetilde{T}_H(x,y)\widetilde{T}_H(1,1)\leq \widetilde{T}_H(x,1)\widetilde{T}_H(1,y).
$$
\end{Lemma}

Note that $\widetilde{T}_H(1,1)=1$, so it appears in the lemma only for aesthetic reasons.

\begin{Lemma} \label{lower-bound}
Let $H$ be an arbitrary bipartite graph, and let $d_i$ be the degree of vertex $i$. Suppose that $0\leq x\leq 1$ and $y\geq 1$ or $0\leq y\leq 1$ and $x\geq 1$. Then
$$\widetilde{T}_H(x,y)\geq \prod_{i\in A}\left(1+\frac{x-1}{d_i+1}\right) \cdot \prod_{j\in B}\left(1+\frac{y-1}{d_j+1}\right).$$
\end{Lemma}

To prove Lemma~\ref{special rectangle} and \ref{lower-bound}, we need the following inequality of Harris that is also a special case of the FKG-inequality \cite{fortuin1971correlation}.

\begin{Lemma}[Harris \cite{harris1960lower}, Fortuin, Kasteleyn, Ginibre \cite{fortuin1971correlation}] \label{FKG-inequality}
Suppose that $\mu$ is the uniform measure on $[0,1]^N$, and $X_1,\dots ,X_t$ are non-negative monotone increasing functions in the sense that if $x_i\geq x_i'$ for $i=1,\dots ,N$, then for $1\leq j\leq t$ we have
$$X_j(x_1,\dots ,x_N)\geq X_j(x_1',\dots ,x_N').$$
Then
$$\E_{\mu}\left[\prod_{j=1}^tX_j\right]\geq \prod_{j=1}^t\E_{\mu}[X_j].$$
Furthermore, if $X$ is monotone increasing and $Y$ is monotone decreasing, then
$$\E[XY]\leq \E[X]\E[Y].$$
\end{Lemma}

In what follows, we repeatedly use the same idea to express $\widetilde{T}_H(x,y)$. This is a crucial idea.

 We can create a random ordering of the vertices of $H$ as follows: for each vertex $i$ we choose a uniform random number $x_i$ from the interval $[0,1]$. The numbers $x_i$ then determine an ordering of the edges. The probability that two numbers are equal is $0$.

 \begin{Lemma}\label{vlsz}
Let $H$ be a bipartite graph and let $\widetilde{T}_H(x,y)=\sum t_{i,j}(H)x^iy^j$. Let $v(H)=m$ and let $x_1, x_2, \dots x_m$ be i.i.d. random variables with distribution $x_i\sim U(0,1)$. Let $I(A)=\left|\left\{v\in A |\  x_v\ge x_{v'} \text{ for } v'\in N_H(v)\right\}\right|$ and  $I(B)=\left|\left\{v\in B |\ x_v\ge x_{v'} \text{ for } v'\in N_H(v)\right\}\right|$. Then

$$\mathbb{P}\left(I(A)=i, I(B)=j\right)=t_{i,j}(H).$$
\end{Lemma}

\begin{proof}
For $\pi \in S_m$, let $E_{\pi}$ be the event that $x_{\pi(1)}> x_{\pi(2)}>\dots>x_{\pi(m)}$. Then we have $\mathbb{P}\left(I(A)=i, I(B)=j| E_{\pi}\right) =1$ if $\ia(\pi)=i$ and $\ea(\pi)=j$, otherwise it is $0$, so by the law of total probability we have
\begin{align*}
    \mathbb{P}\left(I(A)=i, I(B)=j\right)&=\sum_{\pi\in S_m}\mathbb{P}\left(I(A)=i, I(B)=j| E_{\pi}\right)\cdot \mathbb{P}\left(E_{\pi}\right)\\
    &=\frac{1}{m!}\sum_{\pi\in S_m}\mathbb{P}\left(I(A)=i, I(B)=j|\  E_{\pi}\right)\\
    &=t_{i,j}(H).
\end{align*}
\end{proof}

In what follows we do a little trick. For $i\in A$ we generate $x_i\sim U(0,1)$ as before, but for $j\in B$  we actually first generate a uniformly random number $y_j$ from $[0,1]$ and let $x_j=1-y_j$. 
The role of this trick will be apparent soon.
 
For $i\in A$, let us introduce the random variable
$$X_{i}(x_i,\{y_j\}_{j\in B})=\left\{ \begin{array}{ll}
x & \mbox{if}\ \max_{j\in N_H(i)}(1-y_j)\leq x_i,\\
1 & \mbox{if}\ \max_{j\in N_H(i)}(1-y_j)> x_i.
\end{array} \right.$$
and for $j\in B$, let
$$Y_{j}(\{x_i\}_{i \in A},y_j)=\left\{ \begin{array}{ll}
y & \mbox{if}\ \max_{i\in N_H(j)}x_i\leq 1-y_j,\\
1 & \mbox{if}\ \max_{i\in N_H(j)}x_i\geq 1-y_i.
\end{array} \right.$$

\begin{Lemma}
(a) We have 
$$\widetilde{T}_H(x,y)=\E\left[ \prod_{i\in A}X_i\cdot \prod_{j\in B}Y_j\right].$$
(b) If $x\geq 1$, then $X_{i}(x_i,\{y_j\}_{j\in B})$ is a monotone increasing function for each $i\in A$. \\
If $0\leq x\leq 1$, then $X_{i}(x_i,\{y_j\}_{j\in B})$ is a monotone decreasing function for each $i\in A$. \\
For $0\leq y\leq 1$ the function $Y_{j}(\{x_i\}_{i \in A},y_j)$ is monotone increasing for each $j\in B$.\\ Finally, for $0\leq y\leq 1$ the function $Y_{j}(\{x_i\}_{i \in A},y_j)$ is monotone decreasing for each $j\in B$.
\end{Lemma}

\begin{Rem}
This lemma is the reason why we generated $x_j$ by $x_j=1-y_j$ for $j\in B$. If we consider the function 
$$X'_{i}(x_i,\{x_j\}_{j\in B})=\left\{ \begin{array}{ll}
x & \mbox{if}\ \max_{j\in N_H(i)}x_j\leq x_i,\\
1 & \mbox{if}\ \max_{j\in N_H(i)}x_j> x_i.
\end{array} \right.$$
instead of $X_i(x_i,\{y_j\}_{j\in B})$, this  would be neither increasing, nor decreasing. 
\end{Rem}

\begin{proof}
Since we simply generated a uniform random ordering of the vertices, we get that
$$\widetilde{T}_H(x,y)=\E\left[ \prod_{i\in A}X_i\cdot \prod_{j\in B}Y_j\right].$$
We only prove the first statement of part (b), the proof of the other claims are analogous.
Observe that if $x_i,\{y_j\}_{j\in B}$ satisfies that $\max_{j\in N_H(i)}(1-y_j)<x_i$, then increasing $x_i,\{y_j\}_{j\in B}$ cannot ruin this inequality. It can occur though that $\max_{j\in N_H(i)}(1-y_j)<x_i$ previously was not true, but after increasing $x_i,\{y_j\}_{j\in B}$ it becomes true. In this case the value of $X_i(x_i,\{y_j\}_{j\in B})$ jumps from $1$ to $x$, that is, since $x\geq 1$, the value of $X_i(x_i,\{y_j\}_{j\in B})$ is increasing. A similar argument proves the other three statements of part (b).
\end{proof}

Now we are ready to prove Lemma~\ref{special rectangle} and \ref{lower-bound}.

\begin{proof}[Proof of Lemma~\ref{special rectangle}]
If $x\geq 1$ and $0\leq y\leq 1$, then
$\prod_{i\in A}X_i$ and $\prod_{j\in B}Y_j$ are both monotone increasing random variables. Hence by the Harris-inequality (Lemma~\ref{FKG-inequality}), we have
$$\widetilde{T}_H(x,y)=\E\left[ \prod_{i\in A}X_i\cdot \prod_{j\in B}Y_j\right]\geq \E\left[\prod_{i\in A}X_i\right] \cdot \E\left[\prod_{j\in B}Y_j\right]=\widetilde{T}_H(x,1)\widetilde{T}_H(1,y).$$
The other inequalities follow the same way.
\end{proof}

\begin{proof}[Proof of Lemma~\ref{lower-bound}]
We have
$$\E[X_i]=\left(1-\frac{1}{d_i+1}\right)+\frac{x}{d_i+1}=1+\frac{x-1}{d_i+1},$$
and
$$\E[Y_j]=\left(1-\frac{1}{d_j+1}\right)+\frac{y}{d_j+1}=1+\frac{y-1}{d_j+1}.$$
Note that $X_i$ and $Y_j$ are 
monotone increasing functions in terms of the variables $\{x_i\}_{i\in A}$ and $\{y_j\}_{j\in B}$ if $x\geq 1$ and $0\leq y\leq 1$, and they are monotone decreasing functions in terms of the variables $\{x_i\}_{i\in A}$ and $\{y_j\}_{j\in B}$ if $0\leq x\leq 1$ and $y\geq 1$.
Hence, by Harris' inequality (Lemma~\ref{FKG-inequality}), we have
$$\widetilde{T}_H(x,y)=\E\left[ \prod_{i\in A}X_i\cdot \prod_{j\in B}Y_j\right]\geq \prod_{i\in A}\E[X_i] \cdot \prod_{j\in B}\E[Y_j]=\prod_{i\in A}\left(1+\frac{x-1}{d_i+1}\right) \cdot \prod_{j\in B}\left(1+\frac{y-1}{d_j+1}\right).$$

\end{proof}

\begin{Rem}
An interesting application of the above inequalities is the following.  Suppose that a graph $G$ has $n$ vertices, $m$ edges and the length of the shortest cycle is $g$. Then for any spanning tree $T$, the local basis exchange graph $H=H[T]$ has a minimum degree $g-1$ on the side of the non-spanning-tree edges. This means that if $x\geq 1$, we have
$$\widetilde{T}_H(x,0)\geq \widetilde{T}_H(x,1)\widetilde{T}_H(1,0)\geq \widetilde{T}_H(x,1)\prod_{j\in B}\left(1-\frac{1}{d_j+1}\right)\geq \widetilde{T}_H(x,1)\left(1-\frac{1}{g}\right)^{m-n+1}.$$
By summing this inequality for all spanning trees, we get that
$$T_G(x,0)\geq T_G(x,1)\left(1-\frac{1}{g}\right)^{m-n+1}.$$
This inequality is particularly useful if one studies graphs with large girth, and a variant of this inequality was used in the paper \cite{bencs2022evaluations}.
\end{Rem}

\begin{Th}
Let $H$ be a bipartite graph with minimum degree $\delta\geq 1$. Then
$$\widetilde{T}_H\left(2+\frac{1}{\delta},0\right) \widetilde{T}_H\left(0,2+\frac{1}{\delta}\right)\geq \widetilde{T}_H(1,1)^2.$$
In particular, we have
$$\widetilde{T}_H(3,0) \widetilde{T}_H(0,3)\geq \widetilde{T}_H(1,1)^2.$$
Let $G$ be a graph without loops and bridges. Then
$$T_G(3,0) T_G(0,3)\geq T_G(1,1)^2.$$
\end{Th}

\begin{proof}
Let $x=2+\frac{1}{\delta}$. Let us use that $\widetilde{T}_H(1,1)=1$,
$$\widetilde{T}_H(x,0)\geq \prod_{i\in A}\left(1+\frac{x-1}{d_i+1}\right) \cdot \prod_{j\in B}\left(1-\frac{1}{d_j+1}\right),$$
and
$$\widetilde{T}_H(0,x)\geq \prod_{i\in A}\left(1-\frac{1}{d_i+1}\right) \cdot \prod_{j\in B}\left(1+\frac{x-1}{d_j+1}\right).$$
So it is enough to prove that
$$\left(1+\frac{x-1}{d_v+1}\right)\left(1-\frac{1}{d_v+1}\right)\geq 1$$
if $d_v$ is the degree of a vertex $v$.
The inequality $(1+(x-1)\varepsilon)(1-\varepsilon)\geq 1$ is equivalent with $(x-2)\varepsilon\geq (x-1)\varepsilon^2$, that is, $\varepsilon\leq \frac{x-2}{x-1}=\frac{1}{\delta+1}$ which is satisfied since $d_v\geq \delta$ for all vertices $v\in V(H)$. The second inequality follows from the first one by simply taking $\delta=1$.
The third inequality follows from the second one by Lemma~\ref{quadratic-connection}.
\end{proof}

\begin{Cor} \label{simple-co-simple matroids}
If $M$ is a matroid that is simple and co-simple at the same time, then
$$T_M\left(\frac{5}{2},0\right) T_M\left(0,\frac{5}{2}\right)\geq T_M(1,1)^2.$$
\end{Cor}

\begin{proof}
Since $M$ is simple and co-simple at the same time, the minimum degree of any basis exchange graph of $M$ is at least $2$, so we can apply the previous theorem with $\delta=2$. Then the transfer lemma implies the statement. 

\end{proof}

\subsection{Regular bipartite graphs}

In this part we prove that regular bipartite graphs satisfy the inequality $\widetilde{T}_H(2,0)\widetilde{T}_H(0,2)\geq \widetilde{T}_H(1,1)^2$.

\begin{Th}
If $H$ is a regular bipartite graph, then for $x\in [0,2]$, we have
$$\widetilde{T}_H(x,2-x)\geq 1.$$
In particular,
$$\widetilde{T}_H(2,0)\widetilde{T}_H(0,2)\geq \widetilde{T}_H(1,1)^2.$$
\end{Th}

\begin{proof}
We define the random variables $X_i$ and $Y_j$ as before. Recall that if $x\geq 1$ and $0\leq y\leq 1$, then all of these variables are monotone increasing. We can assume that $x\in [1,2]$ and $y=2-x\in [0,1]$.
Since $H$ is regular, it contains a perfect matching $M=\{(u_1,v_1),(u_2,v_2),\dots (u_k,v_k)\}$, where $k=m/2$. Then
$$\widetilde{T}_H(x,y)=\E\left[ \prod_{i\in A}X_i\cdot \prod_{j\in B}Y_j\right]=\E\left[\prod_{i=1}^k\left(X_{u_i}Y_{v_i}\right)\right]\geq \prod_{i=1}^k\E\left[X_{u_i}Y_{v_i}\right]$$
since for each $i$ the random variables $X_{u_i}Y_{v_i}$ are monotone increasing, thus we can use Harris' inequality (Lemma~\ref{FKG-inequality}).
Now observe that
$$\E\left[X_{u_i}Y_{v_i}\right]=1+\frac{x-1}{d+1}+\frac{y-1}{d+1}=1$$
since the probability of $u_i$ being active is $\frac{1}{d+1}$, just as the probability of $v_i$ being active, and these two events exclude each other.
Hence $\widetilde{T}_H(x,y)\geq 1=\widetilde{T}_H(1,1)$.
\end{proof}

\subsection{Trees}

In this section, we prove that trees also satisfy the inequality $\widetilde{T}_H(2,0)\widetilde{T}_H(0,2)\geq \widetilde{T}_H(1,1)^2$. First, we need a lemma about the decompositions of trees.

\begin{Lemma}[Gluing lemma] \label{P(H) of glued trees}
Let $x\geq 1$ and $0\leq y\leq 1$.
Let $H_1$ be a rooted tree with root vertex $v_1$. Let $H_2$ be another rooted tree with root vertex $v_2$. Let $H$ be obtained from $H_1$ and $H_2$ by identifying $v_1$ and $v_2$ in the union of $H_1$ and $H_2$. Let $v$ be the vertex obtained from identifying $v_1$ and $v_2$. Assume that the bipartite parts of $H$ determines the bipartite parts of $H_1$ and $H_2$, that is, if $v\in A(H)$, then $v_1\in A(H_1)$ and $v_2\in A(H_2)$, and if $v\in B(H)$, then $v_1\in B(H_1)$ and $v_2\in B(H_2)$.
\medskip

\noindent (a) If $v\in A$, then
$$x\widetilde{T}_H(x,y)\geq \widetilde{T}_{H_1}(x,y)\widetilde{T}_{H_2}(x,y).$$
\noindent (b) If $v\in B$, then
$$\widetilde{T}_H(x,y)\geq \widetilde{T}_{H_1}(x,y)\widetilde{T}_{H_2}(x,y).$$
\end{Lemma}

\begin{proof}
First, we prove part (b), and the proof of part (a) will be very similar. 

As before, we introduce the random variables $X_i$ and $Y_j$. In particular,
$$Y_{v}(\{x_i\}_{i \in A(H)},y_v)=\left\{ \begin{array}{ll}
y & \mbox{if}\ \max_{i\in N_H(v)}x_i\leq 1-y_v,\\
1 & \mbox{if}\ \max_{i\in N_H(v)}x_i\geq 1-y_v.
\end{array} \right.$$
and similarly,
$$Y_{v_1}(\{x_i\}_{i \in A(H_1)},y_{v})=\left\{ \begin{array}{ll}
y & \mbox{if}\ \max_{i\in N_{H_1}(v)}x_i\leq 1-y_{v},\\
1 & \mbox{if}\ \max_{i\in N_{H_1}(v)}x_i\geq 1-y_{v}.
\end{array} \right.$$
and
$$Y_{v_2}(\{x_i\}_{i \in A(H_2)},y_{v})=\left\{ \begin{array}{ll}
y & \mbox{if}\ \max_{i\in N_{H_2}(v)}x_i\leq 1-y_{v},\\
1 & \mbox{if}\ \max_{i\in N_{H_2}(v)}x_i\geq 1-y_{v}.
\end{array} \right.$$
Note that we think of $H_1$ and $H_2$ as they are embedded into $H$, that is why we used the variable $y_v$ for both $Y_{v_1}$ and $Y_{v_2}$. In this sense, $Y_{v}\geq Y_{v_1}$ because if $Y_v=y$, then $Y_{v_1}=Y_{v_2}=y$ automatically holds. Since $y<1$, we also get that $Y_v\geq Y_{v_1}Y_{v_2}$. Hence
$$\widetilde{T}_H(x,y)=\E\left[ \prod_{i\in A(H)}X_i\cdot \prod_{j\in B(H)}Y_j\right]\geq \E\left[Y_{v_1}Y_{v_2}\prod_{i\in A(H)}X_i\cdot \prod_{j\in B(H)\setminus v}Y_j\right]\geq$$
$$\geq \E\left[ \prod_{i\in A(H_1)}X_i\cdot \prod_{j\in B(H_1)}Y_j\right]\cdot \E\left[ \prod_{i\in A(H_2)}X_i\cdot \prod_{j\in B(H_2)}Y_j\right]=\widetilde{T}_{H_1}(x,y)\widetilde{T}_{H_2}(x,y).$$

In the proof of part (a), we use that $xX_v\geq X_{v_1}X_{v_2}$.  Whence
$$x\widetilde{T}_H(x,y)=\E\left[x \prod_{i\in A(H)}X_i\cdot \prod_{j\in B(H)}Y_j\right]\geq 
\E \left[ X_{v_1}X_{v_2} \prod_{i\in A(H)\setminus v}X_i
\prod_{j\in B(H)}Y_j\right]$$

$$\geq \E\left[\prod_{i\in A(H_1)}X_i\cdot \prod_{j\in B(H_1)}Y_j\right]\E\left[\prod_{i\in A(H_2)}X_i\cdot \prod_{j\in B(H_2)}Y_j\right]=\widetilde{T}_{H_1}(x,y)\widetilde{T}_{H_2}(x,y).$$
\end{proof}

\begin{Def} Let
$$P(H):=\widetilde{T}_{H}(2,0)\widetilde{T}_{H}(0,2).$$
\end{Def}

The following lemma is an immediate consequence of Lemma~\ref{P(H) of glued trees}.

\begin{Lemma} \label{P(H) of glued trees 2}
Let $H_1$ be a rooted tree with root vertex $v_1$. Let $H_2$ be another rooted tree with root vertex $v_2$. Let $H$ be obtained from $H_1$ and $H_2$ by identifying $v_1$ and $v_2$ in the union of $H_1$ and $H_2$. Then
$$P(H)\geq \frac{1}{2}P(H_1)P(H_2).$$
\end{Lemma}

\begin{proof} Suppose that $v\in V(H)$ obtained from identifying $v_1$ and $v_2$. We can assume that $v\in A$ as the argument for $v\in B$ is completely analogous. Then by part (a) of Lemma~\ref{P(H) of glued trees} we have
$$2\widetilde{T}_H(2,0)\geq \widetilde{T}_{H_1}(2,0)\widetilde{T}_{H_2}(2,0).$$
Let $H'=(B,A,E)$ be the graph obtained by switching the two sides of $H$. Then $v\in B(H')$ and by part (b) of Lemma~\ref{P(H) of glued trees} we have
$$\widetilde{T}_{H'}(2,0)\geq \widetilde{T}_{H'_1}(2,0)\widetilde{T}_{H'_2}(2,0).$$
By Lemma~\ref{csere} this is equivalent with
$$\widetilde{T}_{H}(0,2)\geq \widetilde{T}_{H_1}(0,2)\widetilde{T}_{H_2}(0,2).$$
By multiplying the two inequalities we get that
$$2P(H)\geq P(H_1)P(H_2).$$
\end{proof}

Next, we need a lemma that says that we can always decompose a tree into two trees such that none of them is too small or too large.

\begin{Lemma} \label{tree decomposition}
Let $M\geq 2$. Let $H$ be a tree. If $H$ has $M$ edges, then it can be decomposed to edge-disjoint trees $H_1$ and $H_2$ such that both of them have at least $M/3$ edges.
\end{Lemma}

\begin{Rem}
The lemma is tight in the sense that if we have a tree $H $ on $3k+1$ vertices such that from a vertex of degree $3 $ we have $3$ paths of length $k$, then in any decomposition, there is a tree with at most $k$ edges.
\end{Rem}

\begin{proof}[Proof of Lemma~\ref{tree decomposition}]
We give an algorithm to find such a decomposition. If the tree is a path, then the problem is trivial. If the tree is not a path, then let $v$ be a vertex of degree at least $3$. Let $a_1\leq a_2\leq \dots \leq a_k$ be the number of edges of the branches from $v$, that is, $k\geq 3$ is the degree of $v$ and $a_1+\dots +a_k=M$. Clearly, $a_1\leq M/k\leq M/3$. Let us introduce the function $t(v):=a_k$. Let us distinguish three cases. 
\medskip

Case 1: there is an $i<k$ such that $M/3\leq a_1+\dots +a_{i-1}\leq 2M/3$. In this case, we are done because we can put the first $i-1$ branches into $H_1$ and the rest to $H_2$.
\medskip

Case 2: There is an $i<k$ such that $a_1+\dots +a_{i-1}< M/3$ but $a_1+\dots +a_i\geq 2M/3$. Then $a_i>M/3$. Since $i<k$, we have $a_k\geq a_i>M/3$, but then $a_1+\dots +a_k>M$. So this case cannot happen.
\medskip

Case 3: $a_1+\dots +a_{k-1}<M/3$. In this case, let us start to walk in the k-th branch to the next vertex of degree at least $3$, let us call it $u$. So consider the first $k-1$ branches as one in the sequel. Since we are walking on a path, the size of this branch changes one by one. If at some point the size of the branch is at least $M/3$, then we are done. If this is not the case, then we arrive at the next vertex of degree at least $3$, namely $u$, and we can repeat the whole argument. An important observation is though that $t(u)<t(v)$. So by repeating this argument, we eventually arrive at a decomposition where the parts have sizes between $M/3$ and $2M/3$.  
\end{proof}

\begin{Th}
For every tree $H$, we have $P(H)\geq 1$. In fact, if $H$ has at least $10$ vertices, then $P(H)\geq 2$.
\end{Th}

\begin{proof}
By a computer program, we first checked the claim for trees on at most $18$ vertices. Let
$$\Pi(m)=\min_{H\in \mathcal{T}_m}P(H)$$
be the minimum of $P(H)$ among trees on $m$ vertices.

The table at the end of the proof summarizes our findings for $m\leq 18$. One key observation is that $\Pi(m)\geq 2$ for $10\leq m\leq 18$. The other important observation is that
$$\Pi(m_1)\Pi(m_2)\geq 4$$
if $m_1+m_2\geq 20$ and $3\leq m_1,m_2\leq 17$. Let $H$ be a tree with $m$ vertices such that $19\leq m\leq 27$. Then by Lemma~\ref{tree decomposition}, we can decompose it to trees $H_1$ and $H_2$ such that 
$$7\leq \Bigl\lceil\frac{m-1}{3}\Bigr\rceil+1\leq v(H_1),v(H_2)\leq \Bigl\lfloor\frac{2(m-1)}{3}\Bigr\rfloor+1\leq 18.$$
Since $v(H_1)+v(H_2)=m+1\geq 20$, we get that
$P(H)\geq \frac{1}{2}P(H_1)P(H_2)\geq 2$.
So the claim is true for trees on at most $27$ vertices. From now on, we proceed by induction on the number of vertices: we prove that $P(H)\geq 2$ if $H$ has at least $10$ vertices. Let $H$ be a tree on $m$ vertices.  As we have seen, the claim is true if $10\leq  m\leq 27$.  If $m\geq 28$, then we can decompose it into two trees $H_1$ and $H_2$ such that $v(H_1),v(H_2)\geq \frac{m-1}{3}+1\geq 10$, so by induction,
we have 
$P(H)\geq \frac{1}{2}P(H_1)P(H_2)\geq 2$.
This finishes the proof. 
\end{proof}

\bigskip

\begin{center}
\begin{tabular}{|c|c|c|}\hline
$m$ & number of trees & $\Pi(m)$ \\ \hline
2 & 1 & 1 \\ \hline
3 & 1 & 1.3333 \\ \hline
4 & 2 & 1.3611 \\ \hline
5 & 3 & 1.5111  \\ \hline
6 & 6 & 1.5766  \\ \hline
7 & 11 & 1.6585  \\ \hline
8 & 23 & 1.7958  \\ \hline
9 & 47 & 1.8640  \\ \hline
10 & 106 & 2.0589  \\ \hline
11 & 235 & 2.1546  \\ \hline
12 & 551 & 2.3426  \\ \hline
13 & 1301 & 2.4600  \\ \hline
14 & 3159 & 2.5990  \\ \hline
15 & 7741 & 2.8138  \\ \hline
16 & 19320 & 2.9519  \\ \hline
17 & 48629 & 3.1965  \\ \hline
18 & 123867 & 3.3424  \\ \hline
\end{tabular}
\end{center}
\bigskip

\begin{Rem}
The same proof gives that if $m\geq 10$, then
$$\Pi(m)>2\cdot 1.0001^{m-1},$$
so $\Pi(m)$ grows exponentially! It is worth comparing this result with Remark~\ref{complete-bipartite-linear} about balanced complete bipartite graphs.
\end{Rem}

\section{Improvement over $3$}
\label{sect: improvement 3}

This section aims to prove that if $x\geq 2.9243$, then for any bipartite graph $H$ without isolated vertices, we have
$$\widetilde{T}_H(x,0)\widetilde{T}_H(0,x)>\widetilde{T}_H(1,1)^2.$$
From now on let
$$P_x(H)=\widetilde{T}_H(x,0)\widetilde{T}_H(0,x).$$
Clearly, $P(H)$ that was introduced in the previous section is $P_2(H)$ with this notation.

\begin{Lemma} \label{leaf-deletion} Let $H$ be a bipartite graph and $v\in V(G)$.
Let $x\geq 1$ and suppose that $v$ has degree $1$. \\
\noindent (a) If $v\in A$, then
$$\widetilde{T}_H(x,0)\geq \frac{x+1}{2}\widetilde{T}_{H-v}(x,0).$$
\noindent (b) If $v\in B$, then
$$\widetilde{T}_H(x,0)\geq \frac{1}{2}\widetilde{T}_{H-v}(x,0).$$
In particular,
$$P_x(H)\geq \frac{x+1}{4}P_x(H-v).$$
\end{Lemma}

\begin{proof}
Let $u$ be the unique neighbour of $v$. We can think of $H$ as the graph obtained from glueing $K_2$ with $H-v$ at vertex $u$. If $v\in B$, then $u\in A$ and by the gluing lemma (Lemma \ref{P(H) of glued trees}), we have
$$\widetilde{T}_H(x,0)\geq \frac{1}{x}\widetilde{T}_{K_2}(x,0)\widetilde{T}_{H-v}(x,0)=\frac{1}{2}\widetilde{T}_{H-v}(x,0).$$
To prove part (a) of the lemma, we need a strengthening of the glueing lemma: if $u\in B$, then
$$\widetilde{T}_H(x,0)\geq \widetilde{T}_{K_2}(x,1)\widetilde{T}_{H-v}(x,0).$$
This strengthening only works because in $K_2$ there are no other vertices in $A$ apart from $v$. This inequality can be proved as follows. 
As before, we introduce the random variables $X_i$ and $Y_j$. In particular,
$$Y_{u}(\{x_i\}_{i \in A(H)},y_u)=\left\{ \begin{array}{ll}
0 & \mbox{if}\ \max_{i\in N_H(u)}x_i\leq 1-y_u,\\
1 & \mbox{if}\ \max_{i\in N_H(u)}x_i\geq 1-y_u,
\end{array} \right.$$
and similarly,
$$Y'_{u}(\{x_i\}_{i \in A(H-v)},y_{u})=\left\{ \begin{array}{ll}
0 & \mbox{if}\ \max_{i\in N_{H}(u)\setminus \{v\}}x_i\leq 1-y_{u},\\
1 & \mbox{if}\ \max_{i\in N_{H}(u)\setminus \{v\}}x_i\geq 1-y_{u}.
\end{array} \right.$$
Observe that $Y_u\geq Y'_u$ since if $Y_u=0$, then $Y'_u=0$ too. Furthermore, every function is monotone increasing. Then
$$
\widetilde{T}_H(x,0)=\E\left[ \prod_{i\in A(H)}X_i\cdot \prod_{j\in B(H)}Y_j\right] 
=\E\left[ X_vY_u \prod_{i\in A(H) \atop i\neq v}X_i\cdot \prod_{j\in B(H) \atop j\neq u}Y_j\right]
\geq $$
$$\E\left[ X_vY'_u \prod_{i\in A(H) \atop i\neq v}X_i\cdot \prod_{j\in B(H) \atop j\neq u}Y_j\right] \geq \E\left[ X_v\right] \cdot \E\left[Y'_u \prod_{i\in A(H) \atop i\neq v}X_i\cdot \prod_{j\in B(H) \atop j\neq u}Y_j\right]
=\frac{x+1}{2}\widetilde{T}_{H-v}(x,0).
$$

\end{proof}

\begin{Def}
We say that a graph $H$ is a minimal graph with respect to $P_x(\cdot)$ if it contains no induced subgraph $H'$ such that $P_x(H')<P_x(H)$. 
\end{Def}

\begin{Lemma} \label{excluded-subgraph}
Assume that $x>\frac{\sqrt{37}-1}{2}\approx 2.541381265$. Suppose that $H$ is a minimal graph with respect to $P_x(\cdot)$. Then $H$ cannot contain two vertices of degree $1$ connected to the same vertex. Furthermore, if we delete all degree $1$ vertices from $H$, then the obtained graph cannot contain any degree $1$ vertex.
\end{Lemma}

\begin{proof}
For $K_{1,2}$, we have
$$P_x(K_{1,2})=\widetilde{T}_{K_{1,2}}(x,0)\widetilde{T}_{K_{1,2}}(0,x)=\frac{x}{3}\cdot \frac{x^2+x}{3}.$$
If a bipartite graph $H$ contains $K_{1,2}$ as an induced subgraph, then we can think of $H$ as the graph obtained from glueing $K_{1,2}$ with an appropriate graph $H_1$. Then by the gluing lemma (Lemma \ref{P(H) of glued trees}), we have
$$P_x(H)\geq \frac{1}{x}P_x(K_{1,2})P_x(H_1)\geq \frac{x^2+x}{9}P_x(H_1).$$
Since $x>\frac{\sqrt{37}-1}{2}$, the largest zero of $x^2+x-9$, we get that
$P_x(H)>P_x(H_1)$ contradicting the minimality of $H$. Thus $H$ cannot contain $K_{1,2}$ as an induced subgraph, implying both parts of the theorem.
\end{proof}

\begin{Lemma} \label{comparison-with-compbip}
Let $H=(A,B,E)$ be a bipartite graph such that $|A|=a$ and $|B|=b$. If $x,y\geq 1$, then
$$\widetilde{T}_H(x,y)\geq \widetilde{T}_{K_{a,b}}(x,y).$$
\end{Lemma}

\begin{proof}
Given a graph $H$ and a fixed permutation $\pi$, the addition of an edge $(i,j)$ not in $H$ may result in $i$ or $j$ ceased to be active but cannot create a new active vertex. Since $x,y\geq 1$, this means that $\widetilde{T}_H(x,y)$ decreases term by term after the addition of an edge. Hence
$$\widetilde{T}_H(x,y)\geq \widetilde{T}_{K_{a,b}}(x,y).$$

\end{proof}

\begin{Lemma} \label{growth-compbip}
Let $a+b=m$ and $x>2$. Then
$$\widetilde{T}_{K_{a,b}}(x,1) \widetilde{T}_{K_{a,b}}(1,x)\geq \left(\frac{x^2}{4(x-1)}\right)^m.$$
\end{Lemma}

\begin{proof}
We first prove a slightly weaker result, namely, that 
$$\widetilde{T}_{K_{a,b}}(x,1) \widetilde{T}_{K_{a,b}}(1,x)\geq f(m)\left(\frac{x^2}{4(x-1)}\right)^m,$$
where $f(m)$ is a function such that for every $c<1$, we have $f(m)>c^m$ for large enough $m$. Later we remove $f(m)$ with a little trick. In this proof, $f(m)$ will be $m^{-K}$ for some $K$, but we will not specify $K$ as we will remove it anyway.

Recall from Lemma \ref{lem:complete} that
$$\widetilde{T}_{K_{a,b}}(x,y)=\sum_{i=1}^a\frac{a(a-1)...(a-i+1)b}{m(m-1)...(m-i)}x^i+\sum_{j=1}^b\frac{b(b-1)...(b-j+1)a}{m(m-1) \dots (m-j)}y^j.$$
Observe that
$\frac{t_{i+1,0}(K_{a,b})x^{i+1}}{t_{i,0}(K_{a,b})x^i}=\frac{(a-i)x}{m-i-1}.$
So the maximum of $t_{i,0}(K_{a,b})x^i$ is achieved at
$i=\left\lfloor \frac{xa-(m-1)}{x-1}\right\rfloor$
if this number is at least $1$, otherwise the maximal term is $t_{1,0}(K_{a,b})x=\frac{abx}{m(m-1)}.$ For the sake of simplicity, we just use the approximation given by Stirling's formula at $i=\frac{xa-m}{x-1}$ (we omit taking the integer part, this again only affects a polynomial multiplicative error):
$$t_{i,0}(K_{a,b})x^i\approx b\frac{a!(m-i)!}{(a-i)!m!}x^i\approx \frac{\left(\frac{a}{e}\right)^a\left(\frac{m-i}{e}\right)^{m-i}}{\left(\frac{a-i}{e}\right)^{a-i}\left(\frac{m}{e}\right)^m}x^i=\frac{a^a(m-i)^{m-i}}{(a-i)^{a-i}m^m}x^i,$$
where the approximations are up to subexponential terms. (After the first step we built $b$ into the subexponential term.)

Using the notation $a=\alpha m$, we have
$$i=\frac{x\alpha-1}{x-1}m,$$
and the above expression is further approximately equal to
\begin{align*}
\left(\frac{\alpha^{\alpha}\left(1-\frac{x\alpha-1}{x-1}\right)^{1-\frac{x\alpha-1}{x-1}}}{\left(\alpha-\frac{x\alpha-1}{x-1}\right)^{\alpha-\frac{x\alpha-1}{x-1}}}\right)^mx^i
&=\left(\frac{\alpha^{\alpha}(x(1-\alpha))^{x(1-\alpha)/(x-1)}}{(1-\alpha)^{\frac{1-\alpha}{x-1}}(x-1)^{1-\alpha}}\right)^mx^i\\
&=(\alpha^{\alpha}(1-\alpha)^{1-\alpha}x^{(1-\alpha)\frac{x}{x-1}}(x-1)^{\alpha-1})^mx^i\\
&=(\alpha^{\alpha}(1-\alpha)^{1-\alpha}x(x-1)^{\alpha-1})^m.
\end{align*}
So either $\alpha<\frac{1}{x}$ and $\widetilde{T}_{K_{a,b}}(x,0)>\frac{x}{m}$ or 
$\alpha\geq \frac{1}{x}$ and
$$\widetilde{T}_{K_{a,b}}(x,0)=f_1(m)(\alpha^{\alpha}(1-\alpha)^{1-\alpha}x(x-1)^{\alpha-1})^m,$$
where $f_1(m)$ is some subexponential term.

If we introduce the notation $b=\beta m$, then again 
either $\beta<\frac{1}{x}$ and $\widetilde{T}_{K_{a,b}}(0,x)>\frac{x}{m}$ or 
$\beta\geq \frac{1}{x}$ and
$$\widetilde{T}_{K_{a,b}}(0,x)=f_2(m)(\beta^{\beta}(1-\beta)^{1-\beta}x(x-1)^{\beta-1})^m,$$
where $f_2(m)$ is some subexponential term. Note that $\beta=1-\alpha$, \\
$\widetilde{T}_{K_{a,b}}(x,1)=\widetilde{T}_{K_{a,b}}(x,0)+\widetilde{T}_{K_{a,b}}(0,1),$
and $0<\widetilde{T}_{K_{a,b}}(0,1)<\widetilde{T}_{K_{a,b}}(1,1)=1$.

If $\alpha,\beta\geq 1/x$, then
$$\widetilde{T}_{K_{a,b}}(x,1)\widetilde{T}_{K_{a,b}}(1,x)=f_1(m)f_2(m)\left(\alpha^{2\alpha}(1-\alpha)^{2(1-\alpha)}\frac{x^2}{x-1}\right)^m,$$
and the minimum of $\alpha^{2\alpha}(1-\alpha)^{2(1-\alpha)}$ is at $\alpha=\frac{1}{2}$, where it is equal to $\frac{1}{4}$. So we get that 
$$\widetilde{T}_{K_{a,b}}(x,1)\widetilde{T}_{K_{a,b}}(1,x) \geq f(m)\left(\frac{x^2}{4(x-1)}\right)^m,$$
where $f(m)$ is some function that grows faster than any $c^m$ with $c<1$.

Suppose that either $\alpha<\frac{1}{x}$ or $\beta<\frac{1}{x}$. Since $x>2$, it cannot happen that both $\alpha,\beta<\frac{1}{x}$ as their sum is $1$. By symmetry, we can assume that $\beta<\frac{1}{x}$ and $\alpha>1-\frac{1}{x}$. Observe that the function $$h(\alpha):=\alpha^{\alpha}(1-\alpha)^{1-\alpha}x(x-1)^{\alpha-1}$$
is monotone increasing on the interval $\left[\frac{1}{x},1\right]$ as the derivative of its logarithm is \\ $\ln \left(\frac{\alpha}{1-\alpha}\right)+\ln(x-1)$ which is $0$ at $\alpha=\frac{1}{x}$ and positive if $\alpha>\frac{1}{x}$. Since $1-\frac{1}{x}>\frac{1}{x}$ we get that $h(\alpha)\geq h\left(1-\frac{1}{x}\right)=(x-1)^{1-2/x}$. We claim that
$$(x-1)^{1-2/x}\geq \frac{x^2}{4(x-1)}$$
if $x\geq 2$. This is equivalent to the fact that $\frac{1}{x}(x-1)^{1-1/x}\geq \frac{1}{2}$. Since they are equal at $x=2$, it is enough to show that for $x\geq 2$, the left-hand side is monotone increasing. Indeed, taking the derivative of its logarithm we get $\frac{\ln(x-1)}{x^2}\geq 0$. Hence, in this case, it is still true that
$$\widetilde{T}_{K_{a,b}}(x,1)\widetilde{T}_{K_{a,b}}(1,x) \geq f(m)\left(\frac{x^2}{4(x-1)}\right)^m.$$

Now let us take $k$ copies of $K_{a,b}$ and observe that $kK_{a,b}$ is a subgraph of $K_{ka,kb}$. Thus, by the previous lemma, we have
$$(\widetilde{T}_{K_{a,b}}(x,1)\widetilde{T}_{K_{a,b}}(1,x))^k=\widetilde{T}_{kK_{a,b}}(x,1)\widetilde{T}_{kK_{a,b}}(1,x)\geq \widetilde{T}_{K_{ka,kb}}(x,1)\widetilde{T}_{K_{ka,kb}}(1,x)\geq f(km)\left(\frac{x^2}{4(x-1)}\right)^{km}.$$
Hence
$$\widetilde{T}_{K_{a,b}}(x,1)\widetilde{T}_{K_{a,b}}(1,x)\geq f(km)^{1/k}\left(\frac{x^2}{4(x-1)}\right)^{m}$$
for every $k$. Since we have $f(km)>c^{km}$ for every $c<1$ and large enough $k$, we get that
$$\widetilde{T}_{K_{a,b}}(x,1)\widetilde{T}_{K_{a,b}}(1,x)\geq \left(\frac{x^2}{4(x-1)}\right)^{m}.$$

\end{proof}

\begin{Lemma} \label{monotonicity}
If $x\geq 1$, then the function
$$g(k,x):=\frac{\ln\left(1+\frac{x-1}{k+1}\right)}{\ln \left(1+\frac{1}{k}\right)}$$
is a monotone increasing function in $k$. In particular, $g(2,x)\leq g(k,x)$ if $k\geq 2$. 
\end{Lemma}

\begin{proof} Let us differentiate $g(k,x)$ in the variable $k$:
$$\frac{d}{dk}g(k,x)=\frac{(k+x)\ln\left(\frac{k+x}{k+1}\right)-k(x-1)\ln\left(\frac{k+1}{k}\right)}{k(k+1)(k+x)\left(\ln\left(\frac{k+1}{k}\right)\right)^2}.$$
We need to prove that it is positive if $x\geq 1$ and $k\geq 2$. If $x=1$, then the numerator is obviously $0$.  If we differentiate the numerator in the variable $x$, we get that
$$\frac{d}{dx}\left((k+x)\ln\left(\frac{k+x}{k+1}\right)-k(x-1)\ln\left(\frac{k+1}{k}\right)\right)=\ln\left(\frac{k+x}{k+1}\right)+1-k\ln\left(\frac{k+1}{k}\right).$$
This is obviously non-negative since 
$k\ln\left(\frac{k+1}{k}\right)\leq k\cdot \frac{1}{k}=1$. So the numerator is increasing in $x$, so it is always positive for $x\geq 1$ and $k\geq 2$.

\end{proof}

\begin{Th}
If $x\geq 2.9243$, then we have
$$\widetilde{T}_H(x,0)\widetilde{T}_H(0,x)>1$$
for any bipartite graph $H$ without isolated vertices.
\end{Th}

\begin{proof}
Recall that $P_x(H)=\widetilde{T}_H(x,0)\widetilde{T}_H(0,x)$. Suppose for contradiction that for some graph $H$ we have $P_x(H)<1$. We can also assume that $H$ is a minimal graph with respect to $P_x(\cdot)$.

Let $L$ denote the number of leaves of $H$. Let $H'$ be the bipartite graph obtained from $H$ by deleting all these leaves. By Lemma~\ref{leaf-deletion} for $H'$ we have
$$P_x(H)\geq \left(\frac{x+1}{4}\right)^LP_x(H').$$
Observe that we can assume by Lemma~\ref{excluded-subgraph} that we deleted at most one leaf pending from each vertex and that the resulting graph does not contain any leaves. So for the graph $H'=(V',E')$ we have $|V'|\geq L$  and its minimum degree is at least $2$.

By Lemmas \ref{special rectangle} and \ref{lower-bound}, we know that
\begin{align*}
\widetilde{T}_{H'}(x,0)\widetilde{T}_{H'}(0,x)&\geq \widetilde{T}_{H'}(x,1)\widetilde{T}_{H'}(1,x)\widetilde{T}_{H'}(1,0)\widetilde{T}_{H'}(0,1)\\
&\geq \widetilde{T}_{H'}(x,1)\widetilde{T}_{H'}(1,x)\prod_{v\in V'}\left(1-\frac{1}{d_v+1}\right). \\
&\geq \prod_{v\in V'}\left(1+\frac{x-1}{d_v+1}\right)\cdot \prod_{v\in V'}\left(1-\frac{1}{d_v+1}\right).
\end{align*}
Next, we practically distinguish the two cases whether 
$$\left(\prod_{v\in V'}\left(1-\frac{1}{d_v+1}\right)\right)^{-1}=\prod_{v\in V'}\left(1+\frac{1}{d_v}\right)$$
is small or (exponentially) large. If it is large, then we argue that 
$\prod_{v\in V'}\left(1+\frac{x-1}{d_v+1}\right)$ is even larger. If it is small, then we will argue that $\widetilde{T}_{H'}(x,1)\widetilde{T}_{H'}(1,x)$ is still exponentially large. So in both cases, $\widetilde{T}_{H'}(x,0)\widetilde{T}_{H'}(0,x)$ is exponentially large. The details are the following. 

Let 
$$C=\left(\prod_{v\in V'}\left(1+\frac{1}{d_v}\right)\right)^{1/|V'|}.$$
Observe that since every degree is at least $2$ in $H'$, we have
$$\prod_{v\in V'}\left(1+\frac{x-1}{d_v+1}\right)=
\prod_{v\in V'}\left(1+\frac{1}{d_v}\right)^{g(d_v,x)}\geq \left(\prod_{v\in V'}\left(1+\frac{1}{d_v}\right)\right)^{g(2,x)}$$
by Lemma~\ref{monotonicity}. Note that 
$$g(2,x)=\frac{\ln\left(1+\frac{x-1}{3}\right)}{\ln\left(\frac{3}{2}\right)}>1$$
since $x>5/2$. 
Hence we have
$$\widetilde{T}_{H'}(x,0)\widetilde{T}_{H'}(0,x)\geq \prod_{v\in V'}\left(1+\frac{x-1}{d_v+1}\right)\cdot \prod_{v\in V'}\left(1-\frac{1}{d_v+1}\right)\geq C^{(g(2,x)-1)|V'|}.$$
We can, of course, combine Lemma~\ref{comparison-with-compbip} and \ref{growth-compbip} and get that
$$\widetilde{T}_{H'}(x,1)\widetilde{T}_{H'}(1,x)\geq \left(\frac{x^2}{4(x-1)}\right)^{|V'|}.$$
This implies that
$$\widetilde{T}_{H'}(x,0)\widetilde{T}_{H'}(0,x)\geq \widetilde{T}_{H'}(x,1)\widetilde{T}_{H'}(1,x)\prod_{v\in V'}\left(1-\frac{1}{d_v+1}\right)\geq \left(\frac{x^2}{4(x-1)}\right)^{|V'|}\cdot C^{-|V'|}.$$
Now let us consider two cases.

\noindent Case 1. Suppose that $C^{g(2,x)}\geq \frac{x^2}{4(x-1)}$. Then 
$$\widetilde{T}_{H'}(x,0)\widetilde{T}_{H'}(0,x)\geq C^{(g(2,x)-1)|V'|}\geq \left(\frac{x^2}{4(x-1)}\right)^{\frac{g(2,x)-1}{g(2,x)}|V'|}.$$
\noindent Case 2. Suppose that $C^{g(2,x)}\leq \frac{x^2}{4(x-1)}$. Then
$$\widetilde{T}_{H'}(x,0)\widetilde{T}_{H'}(0,x)\geq \left(\frac{x^2}{4(x-1)}\right)^{|V'|}\cdot C^{-|V'|}\geq\left(\frac{x^2}{4(x-1)}\right)^{\frac{g(2,x)-1}{g(2,x)}|V'|}.
$$
So, after all, we get that
$$P_x(H)\geq \left(\frac{x+1}{4}\right)^LP_x(H')\geq \left(\frac{x+1}{4}\right)^L\left(\frac{x^2}{4(x-1)}\right)^{\frac{g(2,x)-1}{g(2,x)}|V'|}.$$
Note that $L\leq |V'|$ and for $x\geq 2.9243$, we have 
$$\frac{x+1}{4}\left(\frac{x^2}{4(x-1)}\right)^{\frac{g(2,x)-1}{g(2,x)}}>1.$$
Hence $P_x(H)>1$, contradiction.

\end{proof}

\section{The alternating function of a bipartite graph} \label{sect: alternating function}

In this section, we study a special coefficient of the permutation Tutte polynomial, which we call the alternating number of the graph.

\begin{Prop} \label{same}
Let $H=(A,B,E)$ be a bipartite graph without isolated vertices. Let $|A|=a$ and $|B|=b$. Then
$$t_{a,0}(H)=t_{0,b}(H).$$
\end{Prop}

\begin{proof}
Let $\pi \in S_m$ be a permutation such that $\ia(\pi)=a$ and $\ea(\pi)=0$. Since every vertex of $A$ is internally active in $\pi$, if $j \in B$ and $i \in N_{H}(j) \subseteq A$, we have $\pi(i)>\pi(j)$. So letting $\pi'(v) = m+1- \pi(v)$, we get a permutation in which every vertex of $B$ is externally active. Clearly, no vertex of $A$ remains internally active in $\pi'$, so $\ia(\pi')=0$ and $\ea(\pi')=b$. The bijection $\pi \mapsto \pi'$ shows that the coefficient of $x^a$ is the same as the coefficient of $y^b$ in $\widetilde{T}(x,y)$.     
\end{proof}

\begin{Cor} \label{is}
Let $H=(A,B,E)$ be a bipartite graph. Let $|A|=a$ and $|B|=b$. If $r$ is the number of isolated vertices in $A$, and $\ell$ is the number of isolated vertices in $B$, then
$$t_{a,\ell}(H)=t_{r,b}(H).$$
\end{Cor}

\begin{proof}

Let $H'$ be the graph we obtain from $H$ by removing all isolated vertices. Then Lemma \ref{comp} implies that $\widetilde{T}_{H}(x,y)=x^ry^{\ell}\widetilde{T}_{H'}(x,y)$. Using Proposition~\ref{same} for $H'$, we get that

$$t_{a-r,0}(H)=t_{0,b-\ell}(H),$$
hence
$$t_{a,\ell}(H)=t_{r,b}(H).$$

\end{proof}

Now we are ready to define the alternating number of a bipartite graph $H$.

\begin{Def}
Let $H=(A,B,E)$ be a bipartite graph. Let $|A|=a$ and $|B|=b$. Let $r$ denote the number of isolated vertices in $A$, and $\ell$ denote the number of isolated vertices in $B$. We define
$$\mathrm{alt}(H):=t_{a,\ell}(H)=t_{r,b}(H)$$
as the alternating number of $H$.
\end{Def}

\begin{Rem}
The notation $\alt(H)$ is originated from the fact that if $H=P_n$, then $\alt(P_n)=\frac{A_n}{n!}$, where $A_n$ denotes the number of alternating permutations. For instance, $A_5=16$. This is in complete accordance with $\alt(P_5)=\frac{2}{15}=\frac{16}{120}$ is the coefficient of $x^3$ and $y^2$ in the polynomial $\widetilde{T}_{P_5}(x,y)$ from Example \ref{example}.
\end{Rem}

\begin{Lemma}\label{altrek}
For any bipartite graph $H$ without isolated vertices,
$$\alt(H)=\frac{1}{m}\sum_{v\in V(A)}\alt(H-v)=\frac{1}{m}\sum_{v\in V(B)}\alt(H-v).$$

\end{Lemma}

\begin{proof}
By Lemma \ref{rek}, we have that 
\begin{align*}
\alt (H)&=t_{a,0}(H) \\
&= \frac{1}{m}\left(\sum_{v\in V(A)}t_{a,0}(H-v) +\sum_{v\in V(B)}t_{a,0}(H-v)\right) \\
&= \frac{1}{m}\sum_{v\in V(B)}t_{a,0}(H-v) \\
&= \frac{1}{m}\sum_{v\in V(B)}\alt (H-v).
\end{align*}
The third inequality holds because $|A(H-v)|<a$ for $v\in V(A)$, so $t_{a,0}(H-v)=0$. The last inequality follows from $|A(H-v)|=a$ and the fact that since $v\in V(B)$ and $H$ does not have any isolated vertices, $H-v$ does not have isolated vertices in $B(H-v)$. 

The proof of the second identity is analogous. 
\end{proof}

\begin{Lemma} \label{22}
For any bipartite graph $H=(A,B,E)$ with $|A|=a$ and $|B|=b$ vertices, we have
$$\widetilde{T}_H\left(x,\frac{x}{x-1}\right)=\mathrm{alt}(H)\frac{x^{m}}{(x-1)^b}.$$
In particular,
$$\widetilde{T}_H(2,2)=\mathrm{alt}(H)2^{m}.$$

\end{Lemma}

\begin{proof}

We use induction on $v(H)$. If $v(H)=1$, then $H=K_1$ and the single vertex is either in $A$, $a=1$, $b=0$ and 
$$\widetilde{T}_H\left(x,\frac{x}{x-1}\right)=x,$$ 
or the single vertex in $B$, $a=0$ and $b=1$, in which case
$$\widetilde{T}_H\left(x,\frac{x}{x-1}\right)=\frac{x}{x-1}.$$
Suppose that $H$ has $m>1$ vertices, and the statement holds for any graph on at most $m-1$ vertices.

First, suppose that $H$ has isolated vertices, and let $v$ be one of them. If $v\in A$, then $$\widetilde{T}_H\left(x,\frac{x}{x-1}\right)=x\widetilde{T}_{H-v}\left(x,\frac{x}{x-1}\right)=x\cdot \mathrm{alt}(H-v)\frac{x^{m-1}}{(x-1)^b}=\alt(H)\frac{x^{m}}{(x-1)^b}.$$
If $v\in B$, then
$$\widetilde{T}_H\left(x,\frac{x}{x-1}\right)=\frac{x}{x-1} \widetilde{T}_{H-v}\left(x,\frac{x}{x-1}\right)=\frac{x}{x-1}\mathrm{alt}(H-v)\frac{x^{m-1}}{(x-1)^{b-1}}=\alt(H)\frac{x^{m}}{(x-1)^b}.$$

If $H$ does not have isolated vertices, then Lemma \ref{rek} and Lemma \ref{altrek} implies that
\begin{align*}
    \widetilde{T}_H\left(x,\frac{x}{x-1}\right)&=\frac{1}{m}\left(\sum_{v\in V(A)}\widetilde{T}_{H-v}\left(x,\frac{x}{x-1}\right)+\sum_{v\in V(B)}\widetilde{T}_{H-v}\left(x,\frac{x}{x-1}\right)\right)\\
    &=\frac{1}{m}\sum_{v\in V(A)}\alt(H-v)\frac{x^{m-1}}{(x-1)^b}+\frac{1}{m}\sum_{v\in v(B)}\alt(H-v)\frac{x^{m-1}}{(x-1)^{b-1}}\\
    &=\alt(H)\frac{x^{m-1}}{(x-1)^b}+\alt(H)\frac{x^{m-1}}{(x-1)^{b-1}}\\
    &=\alt(H)\frac{x^{m}}{(x-1)^b}.
\end{align*}

\end{proof}

\begin{Rem}
A symmetric form of this statement is the following: if $(x-1)(y-1)=1$, then $$\widetilde{T}_H(x,y)=\alt(H)x^ay^b.$$
\end{Rem}

Recall that given edge weights $(w_e)_{e\in E(G)}$, the maximum weight spanning tree is the spanning tree that maximizes $\sum_{e\in E(T)}w_e$. Kruskal's algorithm provides a very fast way to find this tree.
The following proposition connects the quantity $\alt(H)$ with the problem of finding the maximum weight-spanning tree.

\begin{Prop}
Let $G$ be a graph. For each edge $e\in E(G)$, let us choose uniformly randomly an $x_e\in [0,1]$. Let $T$ be a spanning tree with local basis exchange graph $H[T]$. Then the probability that the maximum weight spanning tree with respect to the weights $(x_e)_{e\in E}$ is $T$ is exactly $\alt(H[T])$. 
\end{Prop}

\begin{Prop} \label{volume}
Let $H$ be a bipartite graph on $m$ vertices. Then $\alt(H)$ is the volume of the polytope determined by $0\le t_i\le 1$ $(1\le i\le m)$, $t_i+t_j\le 1$ if $(i,j)\in E(H)$, i.e the independent set polytope of $H$.
\end{Prop}

\begin{proof}
For $i\in A$, let $x_i\sim U(0,1)$ and $y_i=1-x_i$. Then $y_i\sim U(0,1)$. Let $\ell$ be the number of isolated vertices in $B$. From Lemma \ref{vlsz}, we have
\begin{align*}
    \alt(H)&=\mathbb{P}\left(I(A)=|A|, I(B)=\ell\right)\\
    &=\mathbb{P}\left(x_i\ge x_j \text{ if } i\in A, j\in B, (i,j)\in E(H) \right)\\
    &=\mathbb{P}\left(y_i+ x_j\le 1 \text{ if } i\in A, j\in B, (i,j)\in E(H) \right)\\
    &=\mathrm{Vol}\left(0\le t_i\le 1, t_i+t_j\le 1 \text{ if } (i,j)\in E(H) \right).
\end{align*}
\end{proof}

We note that Proposition~\ref{volume} leads to a definition of $\alt(G)$ for not necessarily bipartite graphs, and as Proposition~\ref{Stein-volume} shows it was already studied in the literature.

\begin{Def}
For any graph $G$, we define $\alt(G)=\mathbb{P}\left(x_i+x_j\le 1 \text{ if } (i,j)\in E(G) \right)$, where $x_i\sim U(0,1)$ i.i.d. random variables, or equivalently, \\ $\alt(G)=\mathrm{Vol}\left(\underline{x}\ |\ 0\le t_i\le 1, t_i+t_j\le 1 \text{ if } (i,j)\in E(G) \right)$.
\end{Def}

\begin{Prop}[Steingrimsson \cite{steingrimsson1994decomposition}] \label{Stein-volume}
For any graph $G$ on $n$ vertices which does not have isolated vertices, we have
$$\mathrm{alt}(G)=\frac{1}{2n}\sum_{i=1}^n\mathrm{alt}(G-v_i).$$
\end{Prop}

\begin{Rem} Using the relation in the previous lemma, one can also generalize the second statement of Lemma~\ref{22} to not necessarily bipartite graphs.
\end{Rem}

We end this section with some simple observations.

\begin{Prop}
Let $G$ be a graph on $n$ vertices without isolated vertices. Then
$$\alt(K_n)\leq \alt(G)\leq \alt(S_n).$$
\end{Prop}

\begin{proof}
Observe that adding an edge only increases the number of constraints on the polytope  $\{ \underline{x}\ |\ 0\le x_i\le 1, x_i+x_j\le 1 \text{ if } (i,j)\in E(G)\}$ whence decreasing the volume. So the minimum of $\alt(\cdot )$ is achieved at $K_n$. The maximum is achieved at a graph for which the deletion of any edge results in an isolated vertex. Only a union of star graphs has this property. Indeed, any connected component should be a tree --otherwise, we can delete the edges outside of a spanning tree-- and if a tree is not a star, then it contains a path on $4$ vertices, and we would be able to delete the middle edge. For a union of stars, we have 
$$\alt(S_{n_1}\cup \dots \cup S_{n_k})=\prod_{i=1}^k\frac{1}{n_i}\leq \frac{1}{n}=\alt(S_n),$$
using that $n_i\ge 2$ for $i\in [k]$.
\end{proof}

\begin{Rem}
One can also prove that the minimum value of $\alt(T)$ among trees on $n$ vertices is achieved at the path $P_n$. We omit the proof of this fact as it is somewhat technical.
\end{Rem}

\section{Brylawski's identities} \label{sect: Brylawski}

Brylawski's identities \cite{brylawski1992tutte} describe the linear relationships between the coefficients of the Tutte polynomial. The following version is given in \cite{beke2023short}. 

\begin{Th}[Generalized Brylawski's identity for matroids ]\label{matr} Let $M$ be a matroid on the set $S$. Let $m$ denote the size of $S$ and $r$ the rank of $S$. Let $T_M(x,y)=\sum_{i,j}t_{i,j}(M)x^iy^j$ be the Tutte polynomial of $M$. Then for any integer $h \geq m$, we have
$$\sum_{i=0}^h\sum_{j=0}^{h-i}\binom{h-i}{j}(-1)^jt_{i,j}(M)=(-1)^{m-r}\binom{h-r}{h-m},$$
and for any integer $0 \leq h < m$, we have 
$$\sum_{i=0}^h\sum_{j=0}^{h-i}\binom{h-i}{j}(-1)^jt_{i,j}(M)=0.$$
\end{Th}

Here we give the corresponding statement to $\widetilde{T}_H(x,y)$. Note that since $T_M(x,y)$ can be written as a sum of the permutation Tutte polynomials of the local basis exchange graphs of $G$, this version implies Theorem~\ref{matr}. The proof of Theorem~\ref{matr} in \cite{beke2023short} relies on the fact that $T_M(x,y)$ simplifies to $(x-1)^{r(E)}y^{|E|}$ along the parabola $(x-1)(y-1)=1$. The variant of this statement holds for $\widetilde{T}_H(x,y)$, as Lemma~\ref{22} shows. So to prove Theorem~\ref{br} one only needs to modify the proof given in \cite{beke2023short}. Alternatively, one can prove the statement by induction, this is the proof that we give here.

\begin{Th}[Brylawski's identity]\label{br}
Let $H$ be a bipartite graph with  $m=a+b$ vertices. Let $\widetilde{T}_H(x,y)=\sum t_{i,j}(H)x^iy^j$. Then for any $h<v(H)$, we have
$$\sum_{i=0}^h\sum_{j=0}^{h-i}\binom{h-i}{j}(-1)^jt_{i,j}(H)=0.$$
Furthermore, if $h=v(H)+k$, $k\ge 0$, then we have 
$$\sum_{i=0}^h\sum_{j=0}^{h-i}\binom{h-i}{j}(-1)^jt_{i,j}(H)=(-1)^{b}\alt(H)\binom{b+k}{k}.$$
\end{Th}

\begin{proof}
The proof is by induction on $v(H)$. If $H$ has only one vertex, then the sum becomes $1$ if $v\in A$ and $-h$ if $v\in B$. It is easy to check that these values satisfy the formulae given in the theorem.

Now suppose that we know the identities for bipartite graphs with at most $m-1$ vertices, and let $H$ be any bipartite graph on $m$ vertices. 
First, we look at the case when there is an isolated vertex $v\in A$. Then by Lemma~\ref{comp} we have $t_{i,j}(H)=t_{i-1,j}(H-v)$, and so 
$$ \sum_{i=0}^h\sum_{j=0}^{h-i}\binom{h-i}{j}(-1)^jt_{i,j}(H)=\sum_{i=1}^h\sum_{j=0}^{h-i}\binom{h-i}{j}(-1)^jt_{i,j}(H)$$ $$=\sum_{i=1}^h\sum_{j=0}^{h-i}\binom{h-i}{j}(-1)^jt_{i-1,j}(H-v)=\sum_{i'=0}^{h-1}\sum_{j=0}^{h-1-i'}\binom{h-1-i'}{j}(-1)^jt_{i',j}(H-v).$$
This implies that neither side changes by adding an isolated vertex to $A$.

Now consider the case when there are no isolated vertices in $A$. Let $\ell \geq 0$ denote the number of isolated vertices in $B$. Let $S=\{v_1, v_2, \dots, v_{\ell} \}$ be the set of these isolated vertices. Let $H'$ be the graph induced by the vertex set $V\setminus S$.

Then $t_{i,j}(H)=t_{i,j-\ell}(H')$, so if $h=m+k$, $k\ge 0$, then using first Lemma~\ref{comp} to remove the vertices of $S$ from $H$, and then applying Lemma~\ref{rek} to $H'$ we get that
$$\sum_{i=0}^h\sum_{j=0}^{h-i}\binom{h-i}{j}(-1)^jt_{i,j}(H)=\sum_{i=0}^h\sum_{j=0}^{h-i}\binom{h-i}{j}(-1)^jt_{i,j-\ell}(H')=$$ 
$$=\frac{1}{m-\ell}\sum_{w\in V(H')}\sum_{i=0}^h\sum_{j=0}^{h-i}\binom{h-i}{j}(-1)^jt_{i,j-\ell}(H'-w)=\frac{1}{m-\ell}\sum_{w\in V(H')}\sum_{i=0}^h\sum_{j=0}^{h-i}\binom{h-i}{j}(-1)^jt_{i,j}(H-w).$$
Now we distinguish several cases. First consider the case when $h=v(H)+k$, where $k\geq 0$. Then
we can use the induction hypothesis. We have  $h=(v(H)-1)+(k+1)$, and if $w\in A$, then $|B(H-w)|=|B(H)|$ does not change, and if $w\in B$, then $|B(H-w)|=|B(H)|-1$. So by induction, we have
$$=\frac{1}{m-\ell}\left(\sum_{w\in A}\alt(H-w)(-1)^b\binom{b+k+1}{k+1}+\sum_{w\in B-S}\alt(H-w)(-1)^{b-1}\binom{b+k}{k+1}\right)$$  
$$=\frac{1}{m-\ell}\left(\sum_{w\in A}\alt(H'-w)(-1)^b\binom{b+k+1}{k+1}+\sum_{w\in B-S}\alt(H'-w)(-1)^{b-1}\binom{b+k}{k+1}\right)$$ 
In the last step, we used that adding isolated vertices does not change the value of $\alt(\cdot)$, that is, $\alt(H'-w)=\alt(H-w)$. Next we use Lemma~\ref{altrek}.
$$=\frac{(-1)^b\alt(H')\cdot (m-\ell)}{m-\ell}\cdot \left(\binom{b+k+1}{k+1}-\binom{b+k}{k+1}\right)=(-1)^b\alt(H)\binom{b+k}{k}.$$
If $h\le m-2$, then we can simply use induction.
$$\sum_{i=0}^h\sum_{j=0}^{h-i}\binom{h-i}{j}(-1)^jt_{i,j}(H)=\sum_{i=0}^h\sum_{j=0}^{h-i}\binom{h-i}{j}(-1)^jt_{i,j-\ell}(H')$$ 
$$=\frac{1}{m-\ell}\sum_{w\in H'}\sum_{i=0}^h\sum_{j=0}^{h-i}\binom{h-i}{j}(-1)^jt_{i,j-\ell}(H'-w)=\frac{1}{m-\ell}\sum_{w\in H'}\sum_{i=0}^h\sum_{j=0}^{h-i}\binom{h-i}{j}(-1)^jt_{i,j}(H')$$
$$=\frac{1}{m-\ell}\sum_{w\in H'}0=0.$$
Finally, if $h=m-1$, then we can combine the induction hypothesis with Lemma~\ref{altrek}. (Below we skipped the steps using $\alt(H-w)=\alt(H'-w)$ and $\alt(H')=\alt(H)$ for sake of brevity.)
$$\frac{1}{m-\ell}\sum_{w\in H'}\sum_{i=0}^h\sum_{j=0}^{h-i}\binom{h-i}{j}(-1)^jt_{i,j}(H-w)$$
$$=\frac{1}{m-\ell}\left(\sum_{w\in A}(-1)^b\alt(H-w)+\sum_{w\in B-S}(-1)^{b-1}\alt(H-w)\right)=\alt(H)-\alt(H)=0,$$
which completes the proof.
\end{proof}

\begin{Prop}
If $H$ is connected and $t_{i,j}(H)>0$ for some $i,j$, then $t_{i'j'}(H)>0$ for all $i' \leq i, j' \leq j, (i',j') \neq (0,0)$.
\end{Prop}

\begin{proof}
If $t_{i,j}(H)>0$, then there is a permutation $\pi$, where there are $i$ active vertices in $A$ and $j$ active vertices in $B$. Observe that these vertices have to form an independent subset of $v(H)$ since two active vertices can't be connected. Now let $i'\le i$, $j'\le j$, $(i',j')\ne (0,0)$. We will construct a permutation $\pi '$ that has $i'$ active vertices in $A$ and $j'$ in $B$, which implies that $t_{i',j'}(H)>0$. Let $\pi '$ be the following: assign $m,m-1,\dots,m-(i'+j')+1$ to some of the originally active vertices, such that there are $i'$ in $A$ and $j'$ in $B$. Since they are independent, they will be active. Then always write the next biggest number on a vertex, such that it has a neighbour that already has a number. Since $H$ is connected, we can always continue this until every number is assigned. By the construction of the process, no other vertex will become active. 
\end{proof}

\section{Concluding remarks}
\label{sect: concluding remarks}

Note that there is an extension of Definition~\ref{main-def} for not necessarily bipartite graphs. 

\begin{Def} \label{general-def}
Let $G=(V,E)$ be an arbitrary graph. Suppose that $V=[m]$. For a permutation $\pi$: $[m] \to [m]$, we say that a vertex $i \in V$ is active if $$\pi(i) > \max_{j \in N_G(i)} \pi(j).$$ Let $A(\pi)$ denote the set of active vertices with respect to the permutation $\pi$. We assign the variable $x_i$ to the vertex $i$. Then
$$
\widetilde{T}_G(x_1, x_2, \dots, x_m) = \frac{1}{m!} \sum_{\pi \in S_m} \prod_{i \in A(\pi)} x_i.
$$
\end{Def}

A univariate version of this polynomial was considered in \cite{diaz2019peaks}, though they treated degree $1$ slightly differently. It would be interesting to explore these polynomials in more depth.
\bigskip

In this paper, we proved that if $x_1= 2.9243$, then
$$T_M(x_1,0)T_M(0,x_1)\geq T_M(1,1)^2$$
for every matroid $M$. In \cite{beke2024merino}, it was proved that if $x<x_0$, where $x_0$ is the largest zero of the polynomial $x^3-9x+9$, then there exist matroids for which
$$T_M(x,0)T_M(0,x)< T_M(1,1)^2.$$
We have $x_0\approx 2.22668...$ It would be interesting to close the gap between the lower and the upper bound. We believe that the truth is closer to the lower bound, but it is very unlikely that some variant of our method would yield a bound that goes below $2.5$.

\bigskip

\noindent \textbf{Acknowledgment.} The third author of this paper is very grateful to Ferenc Bencs for useful discussions and that he provided a fast algorithm with implementation to compute the permutation Tutte polynomials of trees.  The authors thank the anonymous referees for their very thorough, constructive and helpful remarks.

\printbibliography

\end{document}